\documentclass{amsart}

\usepackage{etex}
\usepackage{bm}
\usepackage{amssymb}
\usepackage[stretch=10,shrink=10]{microtype}
\usepackage{tikz}
\usepackage{tikz-qtree}
\usepackage{enumerate}
\usepackage{theoremref}
\usepackage{comment}
\usepackage[showframe=false, margin = 1.5 in]{geometry}
\usepackage{pgfplots}
\pgfplotsset{compat=1.5}
\usepackage{hyperref}

\newcommand{\Bin}{\mathrm{Bin}}

\renewcommand{\emptyset}{\varnothing}
\newcommand{\E}{\mathbf{E}}
\renewcommand{\P}{\mathbf{P}}

\newcommand{\ZZ}{\mathbb{Z}}
\newcommand{\NN}{\mathbb{N}}

\newcommand{\TT}{\mathbb{T}}

\renewcommand{\emptyset}{\varnothing}

\newcommand{\Z}{\mathbb Z}
\newcommand{\tf}{\tfrac}

\newcommand{\f}{\frac}
\newcommand{\Aa}{\mathcal A}
\newcommand{\ind}[1]{\mathbf{1}{\{#1\}} }

\usepackage{color}
\definecolor{dark_green}{RGB}{1, 180, 1}

\DeclareMathOperator{\Poi}{Poi}
\DeclareMathOperator{\Ber}{Bernoulli}

\usepackage{microtype}
\usepackage{mathrsfs}
\tikzset{vert/.style={circle,fill,inner sep=0,
    minimum size=0.15cm,draw}, nerve/.style={circle,inner sep=0,
    minimum size=0.15cm,draw},
    rightnerve/.style={circle,inner sep=0,
    minimum size=0.3cm,draw, fill}}

\newcommand{\A}{\mathcal A}

\newcommand{\Tr}{\Lambda}

\newcommand{\Thom}{\mathbb T_d^{\text{hom}}}

\newcommand{\pproc}{\xi}

\newcommand{\rootvertex}{\varnothing}

\newcommand{\homtree}{\TT^{\text{hom}}}

\theoremstyle{plain}
\newtheorem{thm}{Theorem}
\newtheorem{lemma}[thm]{Lemma}
\newtheorem{prop}[thm]{Proposition}
\newtheorem{cor}[thm]{Corollary}
\newtheorem{define}[thm]{Definition}
\newtheorem{conjecture}[thm]{Conjecture}
\newtheorem{question}[thm]{Open Question}
\newtheorem*{claim}{Claim}

\theoremstyle{remark}
\newtheorem{remark}[thm]{Remark}

\tikzset{every tree node/.style={align=center}}
\usepackage{url}

\begin{document}


\title[Recurrence and transience for the frog model on trees]{Recurrence and transience\\ for the frog model on trees}

\author{Christopher Hoffman}
\address{Department of Mathematics, University of Washington}
\email{hoffman@math.washington.edu}

\author{Tobias Johnson}
\address{Department of Mathematics, University of Southern California}
\email{tobias.johnson@usc.edu}

\author{Matthew Junge}
\address{Department of Mathematics, University of Washington}
\email{jungem@math.washington.edu}

\begin{abstract}
The frog model is a growing system of random walks where a particle is added whenever a new site is visited. A longstanding open question is how often the root is visited on the infinite $d$-ary tree. We prove the model undergoes a phase transition, finding it recurrent for $d=2$ and transient for $d\geq 5$. Simulations suggest strong recurrence for $d=2$, weak recurrence for $d=3$, and transience for $d\geq 4$. Additionally, we prove a 0-1 law for all $d$-ary trees, and we exhibit a graph on which a 0-1 law does not hold.

To prove recurrence when $d=2$, we construct a recursive distributional equation for the number of visits to the root in a smaller process and show the unique solution must be infinity a.s. The proof of transience when $d=5$ relies on computer calculations for the transition probabilities of a large Markov chain. We also include the proof for $d \geq 6$, which uses similar techniques but does not require computer assistance.

\end{abstract}

\maketitle

\section{Introduction}


The frog model is a system of interacting random walks on a given rooted graph.
Initially, the graph contains one particle at the root and
some configuration of sleeping particles on its vertices;
unless otherwise stated, we will assume an initial condition of one sleeping particle per vertex.
The particle at the root starts out awake and performs a simple nearest-neighbor random walk in discrete time.
Whenever a vertex with sleeping particles is first visited, all the particles at the site wake up
and begin their own independent random walks, waking particles as they visit them.
A formal definition of the frog model is in \cite{shape}, and a nice survey of variations is in \cite{frogs}.
Traditionally, particles are referred to as frogs, a practice we will continue here.

One of the most basic questions about the frog model on an infinite graph is whether it is 
recurrent or transient. Telcs and Wormald determined that the frog model was recurrent on $\ZZ^d$ for
any $d$, the first published result on the frog model \cite{telcs1999}. 
On an infinite $d$-ary tree, this question is more difficult.
It was first posed in \cite{phasetree}. It was asked again in \cite{frogs} and
in \cite{recurrence}, which pointed out that the answer was unknown even on a binary tree.

Our main result in this paper is that the frog model is recurrent on the binary tree
but transient on the $d$-ary tree for $d\geq 5$, demonstrating a phase transition not found on $\ZZ^d$.
A branching random walk martingale argument proves transience when $d \geq 6$. Pushing this result
down to $d=5$ is more complicated and requires computer assistance.
Our proof of recurrence on the binary tree uses a bootstrapping argument in which we iteratively
assume that the number of visits to the root is stochastically large and then prove it even
larger; the argument seems novel to us.

\subsection*{Background on the frog model}

 It is common to use the frog model as a model for the spread of rumors or infections, thinking of awakened frogs as informed or infected agents. See \cite{daley} for an overview and \cite{epidemic} for more tailored discussion. 
 Another perspective on the frog model is as a conservative lattice gas model with the reaction $A + B \to 2 A$. Here $A$ represents an an active particle and $B$ an inert particle. Active particles disperse throughout the graph, igniting any inert particles they contact. 
Several papers taking this perspective study a process identical to the frog model except that particles
move in continuous rather than discrete time \cite{sid04,CQR,BR}. 
This process and its variants have also seen much study by physicists; see the references in \cite{CQR}
and \cite{BR}. Our results in this paper depend
only on the paths of the frogs and not on the timing of their jumps,
and so they apply equally well to this continuous-time process.

In the larger mathematical context, the frog model is part of a family of self-interacting random walks which have proven quite difficult to analyze. (\cite{pemantlesurvey} provides a nice survey of this family.)
In recent years progress on a few self-interacting random walks has generated considerable interest.  
One of these close relatives is activated random walk, which is touched on
in \cite{KS} and studied in depth in \cite{DRS,RS,ST}.  
Activated random walk can be described as a frog model where frogs fall back asleep
at some given rate. 
Another related process is excited random walk \cite{excited}.
This walk has a bias the first time it visits a site but is unbiased each subsequent time that it returns.
The frog model can be thought of as an ``excited" branching process, which branches at a site $v$ only the first time the process visits $v$.

Initial interest in the frog model was on the graph $\ZZ^d$. For any $d$, it was shown that the process
is recurrent \cite{telcs1999} and that the set of visited vertices grows linearly and when rescaled
converges to a limiting shape \cite{shape}. A similar shape theorem was proven independently in \cite{sid04} 
for the process with continuous-time particles.
Both shape theorems rely on the subadditive ergodic theorem. A technical difficulty that arises is proving that the expected time to wake a given frog is finite. Thus, measuring recurrence on a given graph is an initial step in understanding the long-time behavior of the model. Transience and recurrence continue to attract attention. Also on $\mathbb Z^d$, \cite{random_frog} establishes that the frog model 
undergoes a phase transition from transience to recurrence when the density of frogs decays proportional to distance to the origin. 
Frog models in which frogs move with a bias in one direction are studied on $\mathbb{Z}$ in \cite{recurrence} and \cite{integers_drift2} and on
$\mathbb{Z}^d$ in \cite{dobler2014} and \cite{01frog}.



Our main interest in this paper is in recurrence and transience on $\mathbb T_d$, the infinite rooted $d$-ary tree. We denote the root by $\emptyset$. 
We also study aspects of the process on $\homtree_d$, the
homogeneous degree $(d+1)$-tree, by which we mean the infinite tree where every vertex has degree~$d+1$.



Some attention has already been given to a relative of our model on $\homtree_d$
in which awake frogs die after independently taking a geometrically distributed number of jumps.
In \cite{phasetree} and \cite{po2}, the authors prove a phase transition for survival. Depending on the parameter, there will either be frogs alive at all times with positive probability, or the process will die out almost surely. We study the model in which frogs jump perpetually---a fundamentally different problem, since it switches the emphasis from the local to the global behavior of the model.

\subsection*{Statement and discussion of results}
For a given rooted graph, we call a realization of the model recurrent if the root is visited infinitely many times and transient if it is visited finitely often. 
Our main theorem covers the $d$-ary tree for all but two degrees:

\begin{thm}\thlabel{thm:2tree} \thlabel{thm:5tree} \ \

\begin{enumerate}[(i)]
    \item The frog model on $\TT_2$ is almost surely recurrent.\label{d2}
    \item  The frog model on $\TT_d$ for $d \geq 5$ is almost surely transient.\label{d5}
\end{enumerate}

\end{thm}

We also make a conjecture on these two unsolved degrees
based on fairly convincing evidence from simulations, presented
in Section~\ref{sec:conjectures}. 
\begin{conjecture}\thlabel{conj:phase_transition}
  The frog model on $\TT_d$ is recurrent a.s.\ for $d=3$ and transient 
  a.s.\ for $d= 4$.
\end{conjecture}

The simulations suggest the possibility of a three-phase transition as $d$ increases. 
We call the model strongly recurrent if the probability the root is occupied is bounded away from zero
for all time.
We call it weakly recurrent if it is recurrent with positive probability, but the probability
that the root is occupied decays to zero.

\begin{question} \label{q:strong_weak}
Is the frog model strongly recurrent on $\TT_2$ but weakly recurrent on $\TT_3$?
\end{question}

Such a transition would give information about the time to wake all children of the root. For instance, strong recurrence on $\TT_2$ would imply that this time has finite expectation and an exponential tail. Should $\TT_3$ exhibit a weak recurrence phase, then a tantalizing problem would be to estimate the decay of the occupation time at the root.

The recurrence of the frog model on the binary tree established in \thref{thm:2tree}~(\ref{d2}) is the flagship 
result of this article. 
The proof goes by coupling the frog model with a process in which the root is visited less often. Let $V$ be the number of visits to the root in this restricted model. 
The payoff is \eqref{eqn:keyform}, a recursive distributional equation (RDE, see \cite{AB}) relating $V$ and two independent copies of itself. We find that $V\sim \delta_{\infty}$ is the unique solution. Thus, the original model is recurrent. 

The proof that $V\sim\delta_{\infty}$ is the unique solution of the RDE
uses a seemingly novel bootstrapping argument. 
We assume that $V$ dominates a Poisson  and show that in fact, $V$ dominates a Poisson with slightly
larger mean. It follows by repeating this argument that $V$ dominates any Poisson. 
One obstacle to making this work is that using the typical definition of stochastic dominance,
we cannot establish a base case for the argument.
To get around this, we instead use a weaker stochastic ordering defined in terms of generating functions,
under which the argument holds even when starting with the trivial base case of $V$ dominating 
the distribution $\Poi(0)$.
The situation is different for the frog model with initial conditions of $\Poi(\mu)$ frogs per site.
In this setting, we use the usual notion of stochastic dominance and a related
bootstrapping argument to prove a phase
transition from transience to recurrence on any $d$-ary tree as $\mu$ increases \cite{frog2}.

We believe that these ideas are more widely applicable.
Aldous and Bandyopadhyay study RDEs in general in \cite{AB}. 
Another example of analyzing an RDE
through an induced relation of generating functions can be found in \cite{Liu}.
The RDE~\eqref{eqn:keyform} in this paper is specific to our setting and much more complicated than the RDEs analyzed
in either of these sources. Still, we think that our argument can be applied to other RDEs,
including ones derived from similar interacting particle systems
like activated random walk and the frog model with death.


For the transience part of \thref{thm:2tree}, the idea is to dominate the frog model by a branching process. At the beginning of Section~\ref{sec:transience}, we show in a few lines that a doubling branching random walk is transient on the $14$-ary tree. A simple refinement in \thref{prop:6tree} improves this to $d \geq 6$. The case $d=5$ uses a branching random walk with 27 particle types. This is significantly more complicated,
and computing the transition probabilities requires computer assistance. Conceptually our approach could extend to a computer-assisted proof for transience when $d=4$, but the demands of this theoretical program seem well beyond current processing power. 




We present two other results besides Theorem~\ref{thm:2tree}.
The first is a 0-1 law for transience and recurrence of the frog model on a $d$-ary tree that
applies under more general initial conditions than one frog per site.
For a given distribution $\nu$ on the nonnegative integers, we consider the frog model on a $d$-ary
tree with the number of sleeping frogs on each vertex other than the root drawn independently
from $\nu$. The root initially contains one frog, which begins its life awake. Recall that when a site is
visited for the first time, all sleeping frogs at that site are awoken. We refer to this as
the frog model with i.i.d.-$\nu$ initial conditions. When $\nu=\delta_1$, this is the usual one-per-site
frog model. This theorem complements the 0-1 law for recurrence proven in \cite{recurrence} 
in a frog model on $\Z$ with drift. It also plays an important role in \cite{frog2}, where we use it to show
that the probability of recurrence for the frog model on a $d$-ary tree with i.i.d.-$\Poi(\mu)$
initial conditions jumps abruptly from 0 to 1 as $\mu$ increases. More recently, \cite{01frog} proved a 0-1 law for the frog model that applies in a wide range of circumstances. For instance, it establishes that recurrence holds either with probability~$0$ or $1$ for the frog model on any transitive graph with i.i.d.\ initial
conditions. 
This would apply immediately here, except that we work on $\TT_d$ rather than on $\Thom$.

\begin{thm}\thlabel{thm:01law}
 The frog model on $\TT_d$ for any $d$ and any i.i.d.\ initial conditions is
 recurrent with probability 0 or 1.
\end{thm}

Our final related result is that in contrast to the 0-1 law on $\TT_d$, there is a 
graph on which the frog model is recurrent with probability strictly between $0$ and $1$.

\begin{thm} \thlabel{thm:counterexample}
Let $G$ be the graph formed by merging the root of $\TT_6$ and the origin of $\mathbb Z$ into one vertex. The frog model on $G$ has probability $0<p<1$ of being recurrent. 
\end{thm}

We remark that \cite{random_frog} exhibits a frog model without a 0-1 law on $\ZZ^d$. In their example the
initial distribution of frogs decays in the distance from the origin. 

A few of our proofs would be simplified by changing the setting from $d$-ary to homogeneous trees. However, 
we are interested in applying these results to 
finite trees, and the infinite $d$-ary tree is more natural to work with from that perspective. In any event, the techniques 
underlying our theorems can all be cleanly modified to prove similar 
statements about the homogeneous tree.


\section{Recurrence for the binary tree}\label{sec:recurrence}

 An outline of our proof is as follows. 
 We start by a defining a process that we call the self-similar frog model. A consequence of Proposition~\ref{kermit} is that the number of visits to the root in this model is stochastically 
 smaller than in the original one.
 Thus it suffices to prove the self-similar frog model recurrent. 
    To do this, we define the random variable $V$ to be the number of returns to the root and set $f(x)=\E x^V$, the generating function of $V$. 
    The self-similarity of our model established in \thref{prop:selfsim} allows us to show in Proposition~\ref{mustangsally} that the generating function satisfies the relation $f=\A f$ for an explicit operator $\A$. In \thref{sandinthevaseline}, we show that $\A$ is monotone on a large class of functions. Combining this with $f \leq 1$ on $[0,1]$, we get
 $$f=\A^n f\leq \A^n 1,$$ 
and in \thref{lem:limAa}, we prove that this converges to 0 as $n\to\infty$. 
This implies that $f \equiv 0$ and $V= \infty$ a.s.

This proof can be interpreted as an argument about stochastic orders.
One can define a stochastic order by saying
that if $X$ and $Y$ are nonnegative integer-valued random variables
and $\E t^X\geq \E t^Y$ for $t\in(0,1)$, then $X$ is \emph{smaller in the
probability generating function order} than $Y$.  This order and an equivalent
one called the  Laplace transform order
are discussed 
in \cite[Section~5.A]{SS}.
From this perspective,
each application of the  operator $\Aa$ shows that the distribution
of $V$ is slightly larger in this stochastic order.

  \subsection{The non-backtracking frog model} \label{sec:non-backtracking}
  We will define the \emph{non-backtracking frog model}, in which frogs
  move as random non-backtracking walks stopped at the root.
  More formally, we define the random non-backtracking walk $(X_n,\,n\geq 0)$ as
  a process taking values in $\TT_d$, with
  $X_0=x_0$. On its first step, the walk moves to a uniformly random neighbor of $x_0$.
  At every subsequent step, it chooses uniformly from its neighbors other than the one
  from which it arrived. 
  We emphasize that a non-backtracking walk \emph{can} move towards the root of the tree, though once it moves
  away from the root it will continue doing so.
  Let $T=\inf\{n\geq 1\colon X_n=\rootvertex\}$, taking this to be
  $\infty$ if the walk never visits $\rootvertex$.
  Define the non-backtracking frog model by
  changing the frog's paths in the definition
  of the frog model from simple random walks to the stopped non-backtracking
  walks given by $(X_{n\wedge T},\,n\geq 0)$. Notice that the initial frog is never stopped in this
  model, and only one child of the root is ever visited. Call this child $\rootvertex'$.
  
   \subsection{The self-similar frog model} \label{sec:selfsim}
We make one further alteration to the non-backtracking frog model. Let $\mathbb T_d(v)$ denote the subtree of $\mathbb T_d$ consisting of $v$ and its descendants. Our goal is to make the process viewed on any $\mathbb T_d(v)$ behave identically (in distribution) to the original process. To achieve this, we cap the number of frogs entering $\mathbb T_d(v)$ at one. More formally,  the \emph{self-similar frog model} is the non-backtracking frog model with an additional restriction for each non-root vertex~$v'$ with parent~$v$:
%
  		\begin{itemize}
    		\item Suppose that $v'$ is visited for the first time, necessarily by one or more frogs moving from $v$ to $v'$. Arbitrarily choose all but one of these frogs and stop them at $v'$.
    		\item At all subsequent times, if a frog moves from $v$ to $v'$, stop its path as well.  
    	\end{itemize}
%
%
%
%
The result of this rule is that the number of frogs entering \emph{any} subtree $\mathbb T_d(v')$ is no more than one. 

We now show that in the self-similar frog model, the number of frogs emerging from subtrees activated by a frog is identically distributed for all subtrees. 
   Let $V=V_{\emptyset'}$ be the number of visits to the root in the self-similar
  frog model. Note that only frogs initially sleeping in $\TT_d(\emptyset')$ have a chance of visiting the root. 
  Suppose that vertex~$v$ is visited by a frog. Conditional on this,
  let $v'$ be the child of $v$ that
  the waking frog moves to next, and define $V_{v'}$ as the number of visits to $v$ from frogs in $\TT_d(v')$, the subtree rooted at $v'$.
  \begin{prop} \thlabel{prop:selfsim}
    The distribution of $V_{v'}$ 
    conditional on some frog visiting $v$ and moving next to $v'$
    is equal to the (unconditioned) distribution of $V$.
  \end{prop}
  \begin{proof}
    Let $x$ be the frog that wakes vertex~$v$ and moves from there to $v'$.
   Besides~$x$, all frogs that start outside of $\TT_d(v')$
  get stopped when they try to enter this subtree. Thus, from the time that $v'$ is woken
  on, if we consider the model restricted to $\{v\} \cup \TT_d(v')$,
  it looks identical to the original self-similar frog model (see Figure \ref{fig:selfsim}).
  
  To turn this into a precise statement, consider the model from the time
  $x$ visits $v$ on. Ignore the frog initially at $v$. Freeze frogs when they visit $v$ from
  $\TT_d(v')$. Since no frogs ever enter $\TT_d(v')$, the process depends only
  on the frogs initially in $\TT_d(v')$ and the initial frog~$x$.
  Relabeling vertices $\{v\}\cup\TT_d(v')$ as
   $\{\varnothing\}\cup\TT_d(\varnothing')$ in
  the obvious way then produces a process identically distributed as the original self-similar
  frog model.
  Thus $V$ and $V_{v'}$ are functionals of identically distributed processes.
  \end{proof}

  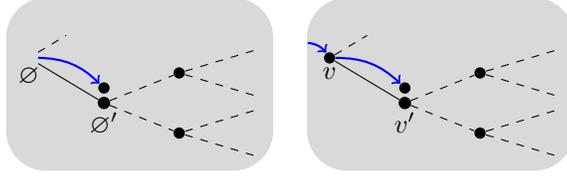
\begin{figure}
    \begin{center}
\begin{tikzpicture}[scale = 1,tv/.style={circle,fill,inner sep=0,
    minimum size=0.15cm,draw}, starttv/.style={circle,draw,inner sep=0,
    minimum size=0.15cm},walklines/.style={very thick, blue,bend left=15}]
    \begin{scope}[shift={(-2,0)}]
    \fill[black!15!white,rounded corners=15pt] (-0.3, -1.5) rectangle (3.25, .8);

      \path (0,0) node (R) {}
        (1,-0.6) node[tv] (L1) {}
        (1, 0.6)  node (L2) {}
        (2, -0.2) node[circle, fill, inner sep = 0, minimum size = .15cm] (L12) {}
        (1, -.4) node[circle, fill, inner sep = 0, minimum size = .15cm] (LU12) {}
        (2,-1) node[circle, fill, inner sep = 0, minimum size = .15cm] (L11) {}
        (L12)+(1,-0.3) coordinate (L121)
             +(1, 0.3) coordinate (L122)
        (L11)+(1,-0.3) coordinate (L111)
             +(1, 0.3) coordinate (L112);
  \draw 
       (R)--(L1);
\node[below] at (R){$\emptyset$};
\node[below] at (L1){$\emptyset'$};
    \draw[dashed] (L1)--(L11);
    \draw[dashed] (R) -- (.5,.3);
       \draw[dashed] (L1) -- (L12);
  \draw [dashed] 
       (L12)--(L122) (L12)--(L121);
   
  \draw[dashed]     (L11)--(L112) (L11)--(L111);
 
\draw[->,thick, blue,bend left=22] (R)  to (LU12);
  \end{scope}

  \begin{scope}[shift={(2,0)}]
 \fill[black!15!white,rounded corners=15pt] (-0.3, -1.5) rectangle (3.25, .8);
    
      \path (0,0) node[circle, fill, inner sep = 0, minimum size = .15cm] (R) {}
        (1,-0.6) node[tv] (L1) {}
        (1, 0.6)  node (L2) {}
        (2, -0.2) node[circle, fill, inner sep = 0, minimum size = .15cm] (L12) {}
        (1, -.4) node[circle, fill, inner sep = 0, minimum size = .15cm] (LU12) {}
        (2,-1) node[circle, fill, inner sep = 0, minimum size = .15cm] (L11) {}
        (L12)+(1,-0.3) coordinate (L121)
             +(1, 0.3) coordinate (L122)
        (L11)+(1,-0.3) coordinate (L111)
             +(1, 0.3) coordinate (L112);
  \draw 
       (R)--(L1);
\node[below] at (R){$v$};
\node[below] at (L1){$v'$};
    \draw[dashed] (L1)--(L11);
    \draw[dashed] (R) -- (.5,.3);
       \draw[dashed] (L1) -- (L12);
  \draw [dashed] 
       (L12)--(L122) (L12)--(L121);
   
  \draw[dashed]     (L11)--(L112) (L11)--(L111);
  
\draw[->,thick, blue,bend left=22] (R)  to (LU12);
\draw[->,thick, blue,bend left=22] (-.29,.2)  to (R);
  \end{scope}
\end{tikzpicture}
    \end{center}
    \caption{Conditional on $v$ being visited, $V$ and $V_{v'}$ are identically distributed in the self-similar model.}
    \label{fig:selfsim}
  \end{figure}

   \subsection{Coupling the models} \label{sec:stubcoupling}
   
   Suppose we wanted to couple a simple and a non-backtracking random walk starting
   from a vertex~$v$ on the 
   homogeneous tree~$\Thom$. Almost surely, there is a 
   unique geodesic from $v$ to infinity that intersects the walk infinitely many times,
   obtained by trimming away the backtracking portions from the walk.
   By symmetry, this geodesic is a uniformly random non-backtracking walk on $\Thom$,
   coupled so that its path is a subset of the simple random walk's path.
   If we were working on $\Thom$ and not $\TT_d$, we could couple the non-backtracking and usual
   frog models as desired by coupling each frog in this way. To address the asymmetry 
   of $\TT_d$ at its root,
   our coupling of non-backtracking and normal frogs on $\TT_d$ will involve an intermediate coupling with
    a random walk on $\Thom$.

  \begin{prop} \label{kermit}
    There is a coupling of the non-backtracking, the self-similar  and the usual frog models so that
    the path of every non-backtracking (self-similar) frog  is a subset of the path of the
    corresponding frog in the usual model.
  \end{prop}
  \begin{proof}
    First, we couple a non-backtracking walk to a simple random walk not on $\TT_d$, but 
    on $\homtree_d$.
    Let $(Y_n,\,n\geq 0)$ be a simple random walk on $\homtree_d$ starting at $x_0$. This random walk
    diverges almost surely to infinity, and there is a unique geodesic from $x_0$ to the path's
    limit. Let $(X_n,n\geq 0)$
    be the path of this geodesic. By the symmetry of $\homtree_d$, the process $(X_n)$ is
    a random non-backtracking walk from $x_0$.
        
    Next, we consider $\TT_d$ as a subset of $\homtree_d$ and
    define a new random walk $(Z_n,\,n\geq 0)$ by modifying $(Y_n)$ as follows.
    First, delete all excursions of $(Y_n)$ away
    from $\TT_d$. This might leave the walk sitting at the root for consecutive steps;
     if so,
     we replace all consecutive occurrences of the root by a single one.
     This results in either an infinite path on $\TT_d$ or a finite path on $\TT_d$ truncated at a visit to the root. In the second case, we extend the path by tacking on an independent simple random walk to its end.
     It follows from the independence of excursions in simple random walk that the
     resulting process $(Z_n)$ is a simple random walk on $\TT_d$.
    
    Let 
    $T$ be the first time past~$0$ that $(X_n)$ hits the root, or $\infty$ if it never does.
    By our construction, $\{X_0,\ldots,X_T\}\subseteq \{Z_n,\,n\geq 0\}$.
    Thus we have coupled the stopped non-backtracking walks and simple random walks
    on $\TT_d$. Coupling each frog in the non-backtracking frog model to the corresponding
    frog in the usual model gives the desired coupling between the non-backtracking and usual frog models.
    As the self-similar model is obtained by stopping frogs in the non-backtracking model,
    we obtain a coupling for it as well.
  \end{proof}

  \subsection{Generating function recursion}

We now apply the self-similarity described in \thref{prop:selfsim} to obtain a relation satisfied by the generating function for the number of visits to the root in the self-similar model.

\begin{define}
    Let $V$ be the number of visits to the root in the self-similar frog model on $\mathbb T_2$. Define $f\colon[0,1] \to [0,1]$ by $f(x) = \E x^V$ with the convention that if $V = \infty$ a.s.\ then $f(1) =0$. 
\end{define}

    \begin{figure}
    \begin{center}
\begin{tikzpicture}[scale = 1,tv/.style={circle,fill,inner sep=0,
    minimum size=0.15cm,draw}, starttv/.style={circle,fill,inner sep=0,
    minimum size=0.15cm},walklines/.style={very thick, blue,bend left=15}]
\fill[black!15!white,rounded corners=15pt] (-0.3, -1.5) rectangle (3.25, 1.2);
    \path (0,0) node[starttv] (R) {}
        (1,-0.6) node[tv] (L1) {}
        (1, 0.6)  node (L2) {}
        (2, -0.2) node[circle, draw, inner sep = 0, minimum size = .15cm] (L12) {}
        (2,-1) node[circle, fill, inner sep = 0, minimum size = .15cm] (L11) {}
        (L12)+(1,-0.3) coordinate (L121)
             +(1, 0.3) coordinate (L122)
        (L11)+(1,-0.3) coordinate (L111)
             +(1, 0.3) coordinate (L112);
  \draw 
       (R)--(L1);
 \path[every node/.style={font=\sffamily\small}]
       (L1) edge [->,bend right,blue] node[above] {\;\;$\textcolor{black}{V}$} (R)
       (L11) edge [->,bend left,blue] node[below] {\;$\textcolor{black}{V_v}$} (L1)
       (L12) edge [->,bend right,blue] node[above] {\;$\textcolor{black}{V_u}$} (L1);
\node[below] at (R){$\emptyset$};
\node[below] at (L12){$u$};
\node[below] at (L11){$v$};
\node[below] at (L1){$\emptyset'$\;\;\;};
    \draw (L1)--(L11);
    \draw[dashed] (R) -- (.5,.3);
       \draw[dashed] (L1) -- (L12);
  \draw[dashed]  
       (L12)--(L122) (L12)--(L121);
   
  \draw     (L11)--(L112) (L11)--(L111);

\end{tikzpicture}
    \end{center}
    \caption{$V$ is the total number of visits to $\emptyset$ in the self-similar process, $V_v$ and $V_u$ are the number of visits to $\emptyset'$ from frogs originally in $\mathbb T_2(v)$ and $\mathbb T_2(u)$, respectively. In the self-similar model $V, V_v$, and $V_u \mid \{\text{$u$ is visited} \}$ are identically distributed}
    \label{fig:ssf}
  \end{figure}
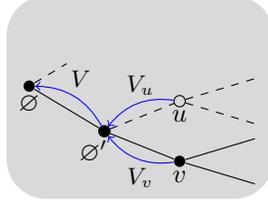
  
  \begin{prop}\label{mustangsally}
    Define $\Aa$, an operator on functions on $[0,1]$, by
  \begin{align}
    \Aa g(x) &=
      \frac{x+2}{3}g\Bigl(\frac{x+1}{2}\Bigr)^2 
      +\frac{x+1}{3}g\Bigl(\frac{x}{2}\Bigr)\biggl(1
        -g\Bigl(\frac{x+1}{2}\Bigr)\biggr).\label{eq:Aadef}
  \end{align}
  The generating function~$f$ satisfies $f=\Aa f$. 
    \end{prop}

 \begin{proof}

   If $\P[V= \infty] =1$ then $f \equiv 0$. This is easily checked to be a fixed point of $\Aa$. So, for the remainder of the argument suppose that $\P[V= \infty] <1$.
  
  The frog at the root in the self-similar model follows a non-backtracking path and visits one of its children and then one of this child's children;
  call these vertices $\emptyset'$ and $v$, respectively. 
  We label the yet to be visited child $u$ (see Figure~\ref{fig:ssf}). 
   Define $V_v$ and $V_u$ to be the number of frogs which visit $\emptyset'$ that were originally sleeping in $\mathbb T_2(v)$ and $\mathbb T_2(u)$, respectively.
   
    \thref{prop:selfsim} guarantees that, since $v$ has been visited, the random variable $V_v$ is distributed identically to $V$. 
   Conditionally on $u$ being visited, the random variable $V_u$ is also distributed identically to $V$.
   In fact, because frogs outside of $\TT_2(u)$ affect $\TT_2(u)$ only by determining whether or not $u$
   is visited, we can express $V_u$ as $V_u = \ind{\text{$u$ is visited}}V'$,
   where $V'$ is distributed as $V$ and is independent of $V_v$. 
   This yields a description of $V$ in terms of a pair of independent copies of itself:
   \begin{align}
V &=   \underbrace{\ind{\text{frog at $\emptyset'$ visits $\emptyset$}}}_{\text{term 1}}  +
   \underbrace{\ind{u \text{ is visited}}  \Bin(V', \tf 12)}_{\text{term 2}} + 
   \underbrace{\Bin(V_v,\tf 12)}_{\text{term 3}}  \label{eqn:keyform}.
\end{align}
Term~1 accounts for a possible visit to $\emptyset$ by the frog started at $\emptyset'$. The conditional binomial distributions in terms~2 and 3 arise because each frog that visits $\emptyset'$ from $u$ or $v$ has a $\f 12$ chance of jumping back to $\emptyset$.  
 
Despite the independence between $V_v$ and $V'$, the three terms are dependent. For example, if term~1 is zero, then term~2
 is more likely to be nonzero, since the frog at $\emptyset'$ not visiting $\emptyset$ makes it more
 likely to visit $u$.
We unearth the pairwise independence of $V_v$ and $V'$ from \eqref{eqn:keyform} by conditioning on the following three disjoint events (see Figure~\ref{fig:ABC}):
  \begin{enumerate}
     \item[$A$.] the frog starting at $\emptyset'$ visits $u$;\label{item:1}
    \item[$B$.] the frog at $\emptyset'$ does not visit $u$,\label{item:2}
      and a frog returns to $\emptyset'$ through $v$ and visits $u$;
    \item[$C$.] no frog ever visits $u$.\label{item:3}
  \end{enumerate}

   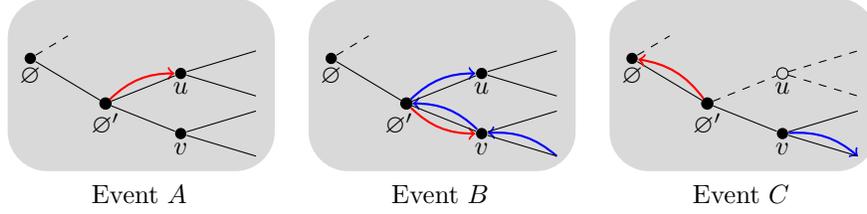
\begin{figure}
    \begin{center}
\begin{tikzpicture}[scale = 1,tv/.style={circle,fill,inner sep=0,
    minimum size=0.15cm,draw}, starttv/.style={circle,fill,inner sep=0,
    minimum size=0.15cm},walklines/.style={very thick, blue,bend left=15}]
    \begin{scope}[shift={(-4,0)}]
    \fill[black!15!white,rounded corners=15pt] (-0.3, -1.5) rectangle (3.25, .8);
    \node at (1.45,-1.8){Event $A$};
    \path (0,0) node[starttv] (R) {}
        (1,-0.6) node[tv] (L1) {}
        (1, 0.6)  node (L2) {}
        (2, -0.2) node[circle, fill, inner sep = 0, minimum size = .15cm] (L12) {}
        (2,-1) node[circle, fill, inner sep = 0, minimum size = .15cm] (L11) {}
        (L12)+(1,-0.3) coordinate (L121)
             +(1, 0.3) coordinate (L122)
        (L11)+(1,-0.3) coordinate (L111)
             +(1, 0.3) coordinate (L112);
  \draw (R)--(L1);
    \draw[->,thick, red,bend left=22] (L1) to (L12);
\node[below] at (R){$\emptyset$};
\node[below] at (L12){$u$};
\node[below] at (L11){$v$};
\node[below] at (L1){$\emptyset'$};
    \draw (L1)--(L11);
    \draw[dashed] (R) -- (.5,.3);
       \draw (L1) -- (L12);
  \draw  
       (L12)--(L122) (L12)--(L121);
   
  \draw     (L11)--(L112) (L11)--(L111);
  \end{scope}
  \begin{scope}[shift={(0,0)}]  
  \fill[black!15!white,rounded corners=15pt] (-0.3, -1.5) rectangle (3.25, .8);  
  \node at (1.45,-1.8){Event $B$};
      \path (0,0) node[starttv] (R) {}
        (1,-0.6) node[tv] (L1) {}
        (1, 0.6)  node (L2) {}
        (2, -0.2) node[circle, fill, inner sep = 0, minimum size = .15cm] (L12) {}
        (2,-1) node[circle, fill, inner sep = 0, minimum size = .15cm] (L11) {}
        (L12)+(1,-0.3) coordinate (L121)
             +(1, 0.3) coordinate (L122)
        (L11)+(1,-0.3) coordinate (L111)
             +(1, 0.3) coordinate (L112);
  \draw (R)--(L1);
    \draw[->,thick, blue,bend left=22] (L1) to (L12);
        \draw[->,thick, blue,bend right=22] (L11) to (L1);
                \draw[->,thick, blue,bend right=22] (L111) to (L11);
    \draw[->,thick, red,bend right=22] (L1) to (L11);
\node[below] at (R){$\emptyset$};
\node[below] at (L12){$u$};
\node[below] at (L11){$v$};
\node[below] at (L1){$\emptyset'$\;\;};
    \draw (L1)--(L11);
    \draw[dashed] (R) -- (.5,.3);
       \draw (L1) -- (L12);
  \draw (L12)--(L122) (L12)--(L121);
  \draw (L11)--(L112) (L11)--(L111);
  \end{scope}
\begin{scope}[shift={(4,0)}]
\fill[black!15!white,rounded corners=15pt] (-0.3, -1.5) rectangle (3.25, .8);
    \node at (1.45,-1.8){Event $C$};
      \path (0,0) node[starttv] (R) {}
        (1,-0.6) node[tv] (L1) {}
        (1, 0.6)  node (L2) {}
        (2, -0.2) node[circle, draw, inner sep = 0, minimum size = .15cm] (L12) {}
        (2,-1) node[circle, fill, inner sep = 0, minimum size = .15cm] (L11) {}
        (L12)+(1,-0.3) coordinate (L121)
             +(1, 0.3) coordinate (L122)
        (L11)+(1,-0.3) coordinate (L111)
             +(1, 0.3) coordinate (L112);
  \draw 
       (R)--(L1);
        \draw[->,thick, blue,bend left=22] (L11) to (L111);
            \draw[->,thick, red,bend right=22] (L1) to (R);
\node[below] at (R){$\emptyset$};
\node[below] at (L12){$u$};
\node[below] at (L11){$v$};
\node[below] at (L1){$\emptyset'$};
    \draw (L1)--(L11);
    \draw[dashed] (R) -- (.5,.3);
       \draw[dashed] (L1) -- (L12);
  \draw [dashed] 
       (L12)--(L122) (L12)--(L121);
   
  \draw     (L11)--(L112) (L11)--(L111);
  \end{scope}
\end{tikzpicture}
    \end{center}
    \caption{Outcomes that would result in events $A$, $B$ and $C$, respectively. The path of the frog at $\emptyset'$ is red and the path of a frog from the subtree $\TT(v)$ is blue.}
    \label{fig:ABC}
  \end{figure}

  Event~$A$ occurs with probability $1/3$.
  Given that $k$ frogs return to $\emptyset'$ through $v$, the probability of $C$
  is $(2/3)2^{-k}$. Since the number of frogs returning to $\emptyset'$ through $v$ is
  distributed identically to $V$, the probability of $C$ is
    $\frac23\E\left(\frac 12\right)^{V}$,
  which we call $2q/3$. 
  Event~$B$ then occurs with the remaining probability, which is $1-1/3-2q/3=2(1-q)/3$. 
  Note that under our assumption $\P[V=\infty] <1$ it follows that $0<q<1$. 
  
  Conditional on event~$A$, $B$, or $C$, the terms in \eqref{eqn:keyform} are independent.
  Indeed, conditioning on whether $u$ is visited makes terms~2 and~3 independent,
  and conditioning
  further on whether the frog at $\emptyset'$ visits $u$ then makes term~1 independent of the other two.
  Now, we describe the distributions of each term in \eqref{eqn:keyform} conditional on events
  $A$, $B$, and $C$. 
  For a given random variable $X$, we use $\Bin(X, p)$ to denote the random variable
  $\sum_{i=1}^X B_i$, where $\{B_i\}_{i\in\NN}$ are distributed as $\Ber(p)$,
  independent of each other and of $X$.
  
  \begin{itemize}
    \item Conditional on $A$,
      term~1 is $0$ and terms~2 and~3 are distributed as independent $\Bin(V, 1/2)$.
    \item Conditional on $B$, term~1 is $\Ber(1/2)$, term~2 is $\Bin(V,1/2)$, and term~3 is 
      $\Bin(V, 1/2)$ conditional on being strictly less than $V$ (since at least one frog will visit $u$ and not move to $\emptyset$). 
        \item Conditional on $C$, term~1 is $\Ber(1/2)$, term~2 is 0, and term~3 is $\Bin(V, 1/2)$ conditional
      on being equal to $V$ (since every frog counted by $V_v$ will return to $\emptyset$).
  \end{itemize}
  To summarize, let $X'$ and $X$ be distributed as $\Bin(V, 1/2)$.
  Let $Y$ be distributed as $\Bin(V, 1/2)$ conditional on $\Bin(V, 1/2)<V$.
  Let $Z$ be distributed as $\Bin(V, 1/2)$ conditional on $\Bin(V, 1/2)=V$.
  Let $I\sim\Ber(1/2)$. Take all of these to be independent.
  Conditioning on events $A$, $B$, and $C$, equation \eqref{eqn:keyform} yields
  \begin{align}
    V &\overset{d}= \begin{cases}
          X' + X & \text{with probability~$1/3$,}\\
          I + X' + Y & \text{with probability~$2(1-q)/3$,}\\
          I + Z & \text{with probability~$2q/3$.}
    \end{cases}\label{eq:Vrecursion}
  \end{align}
  From this description of the distribution of $V$,
  \begin{align}
  \E x^V
  &=\frac13 \E x^{X' + X }+\frac{2(1-q)}{3} \E x^{I + X' + Y}+\frac{2q}3 \E x^{I + Z} \nonumber\\
  &=\frac13 \E x^{X'}  \E x^{  X } +\frac{2(1-q)}{3} \E x^{I} \E x^{X'} \E x^Y+\frac{2q}3 \E x^{I}  \E x^Z.
    \label{cyclones}
  \end{align}

Recall that a $\Ber(p)$ random variable has generating function $px + 1-p$ and that a random sum of i.i.d.\ random variables, $\sum_1^N X_i$, has generating function $g_N(g_{X_1}(x))$, where $g_N$ and $g_{X_1}$ are the generating functions of $N$ and $X_1$. From these facts,
  \begin{align*}
    \E x^{I} &= \frac{x+1}{2},\\
    \E x^{X'}&=\E x^{X} = 
     f\biggl(\frac{x+1}{2}\biggr).
  \end{align*}
  The generating functions $\E x^Y$
  and $\E x^Z$ are a bit more complicated. 
%
The random variable $Y$ is distributed as $X$ conditional on $X <V$. Using the basic formula for conditional probability,
  \begin{align*}
    \P[Y = k] = \P[X = k\mid X<V]
      &= \frac{ \P[X = k \text{ and } X < V] }{\P[X < V] } \\
      &= \frac{\P[X = k] - \P[X = V = k]}
          {1-q}\\
      &= \frac{\P[X=k] - 2^{-k}\P[V=k]}{1-q}.
  \end{align*}
  Thus, the probability generating function of $Y$ is
  \begin{align}
    \E x^{Y} 
      &= \frac{1}{1-q}\sum_{k=0}^{\infty} x^k\bigl(\P[X=k] - 2^{-k}\P[v=k]\bigr)\nonumber\\
      &= \frac{1}{1-q}\E\Bigl[ x^{X} - \Bigl(\frac{x}{2}\Bigr)^{V}\biggr]\label{eqn:switch}\\
      &= \frac{1}{1-q}\biggl( f\Bigl(\frac{x+1}{2}\Bigr) - f\Bigl(\frac{x}{2}\Bigr)
        \biggr).\nonumber
  \end{align}
  In \eqref{eqn:switch} we are making use of the general fact that $\sum ( a_n - b_n) = \sum a_n - \sum b_n$ so long as each sum is finite. Similarly, 
  \begin{align*}
    \P[Z = k] &= \P[X = k\mid X=V]= \frac{2^{-k}\P[V=k]}{q},
  \end{align*}
  and so
  \begin{align*}
    \E x^Z &= \frac{1}{q}\sum_{k=0}^{\infty}x^k2^{-k}\P[V=k] = \frac{1}{q}f\Bigl(\frac{x}{2}\Bigr).
  \end{align*}
  
  Using all of these generating functions and (\ref{cyclones})
  \begin{align*}
    f(x) 
    &=\frac13 \E x^{X'} \E x^{  X }+\frac{2(1-q)}{3} \E x^{I} \E x^{X'} \E x^Y+\frac{2q}3 \E x^{I} \E x^Z\\
      &= \frac13 f\Bigl(\frac{x+1}{2}\Bigr)^2
        + \frac{2(1-q)}{3}\biggl(\frac{x+1}{2}f\Bigl(\frac{x+1}{2}\Bigr)
          \frac{1}{1-q}\Bigl(f\Bigl(\frac{x+1}{2}\Bigr) - f\Bigl(\frac{x}{2}\Bigr)\Bigr)\biggr)\\
          &\qquad\qquad +\frac{2q}{3}\biggl(\frac{x+1}{2q}f\Bigl(\frac{x}{2}\Bigr)\biggr)\\
          &=\frac{x+2}{3}f\Bigl(\frac{x+1}{2}\Bigr)^2 - \frac{x+1}{3}f\Bigl(\frac{x+1}{2}\Bigr)
            f\Bigl(\frac{x}{2}\Bigr) + \frac{x+1}{3} f\Bigl(\frac{x}{2}\Bigr) 
            = \Aa f (x),
  \end{align*}
  which establishes our claim.    
\end{proof}

  

  \subsection{Proving recurrence}
  We have reduced the problem to understanding the properties of the operator~$\Aa$ defined in \eqref{eq:Aadef}.
  In Lemma~\ref{sandinthevaseline}, we prove that $\Aa$ is monotonic for functions belonging to the set $\mathscr S = \{g \colon [0,1] \to [0,1], \text{ nondecreasing}\}$.
  In \thref{lem:preserveincreasing}, we show that $\Aa$ maps $\mathscr S$ into itself,
  so that we can apply Lemma~\ref{sandinthevaseline} after applying $\Aa$ iteratively.
  Finally,  we show in \thref{lem:a+ca,lem:limAa} that $\Aa^n1 \to 0$.
  Starting at the conclusion of Proposition~\ref{mustangsally} (that the generating function~$f$
  is a fixed point of $\Aa$), we will then apply these results to show
   that $f \equiv 0$,
   thus proving that the number of visits to the root in the self-similar frog model is a.s.~ infinite.

  \begin{lemma} \thlabel{sandinthevaseline}
    Let $g,h \in \mathscr S$. If $g\leq h$, then $\Aa g\leq \Aa h$.
  \end{lemma}

\begin{proof}
For $0 \leq t \leq 1$ define the interpolation between $g$ and $h$ by 
$$i_t(x)=(1-t)\cdot g(x)+t\cdot h(x).$$ 
Since $\A i_0 = \A g$ and $\A i _1= \A h $ it suffices to prove that $\f d {dt}\A i_t(x) \geq 0$. 
Fix $x$ and set $a=i_t\bigl(\frac{x+1}{2}\bigr)$  and  $b=i_t\bigl(\frac{x}{2}\bigr)$ so that
$$\A i_t(x) = \frac{2+x}{3}a^2 +\frac{1+x}{3}b(1-a).$$
Define $s(a,b) = \A i_t(x)$. The chain rule implies
\begin{align*}
\frac{d}{dt}\Aa i_t(x) &= \frac{\partial }{\partial a}s(a,b) \frac{da}{dt}+ \frac{\partial }{\partial b}s(a,b)\frac{db}{dt}.
\end{align*}

To prove $\f d{dt} \A i_t \geq 0$ it suffices to prove each term in the above formula is nonnegative. 

\begin{itemize}

    \item  The assumption that $g \leq h$ implies that $\f d {dt} i_t(x) = h(x) - g(x)  \geq 0$ for all $t$ and $x$. In particular, this implies
$\frac{da}{dt},\frac{db}{dt} \geq 0$.
\item First, we compute the partials 
$$ \ \ \frac{\partial }{\partial a}s(a,b)=2a\frac{2+x}{3}-b\frac{1+x}{3} \qquad  \ \    \text{and} \qquad  \ \ 
\frac{\partial }{\partial b}s(a,b) = (1-a)\frac{1+x}{3}.
$$
As $g$ and $h$ are nondecreasing, $i_t$ is also nondecreasing in $x$ for any fixed $t$. Hence $b \leq a$. Along with the bound $a \leq 1$, this immediately implies both partials are positive.\qedhere
\end{itemize}
\end{proof}
\begin{lemma} \thlabel{lem:preserveincreasing}
  If $g \in \mathscr S$, then $\Aa g \in \mathscr S$.
\end{lemma}
\begin{proof}
  All summands in \eqref{eq:Aadef} are nonnegative when $g(x)\leq 1$, which implies
  that $\Aa g\geq 0$. By the previous lemma,
  $\Aa g\leq \Aa 1 \leq 1$. We can conclude then that $0 \leq \A g \leq 1$. To see that $\Aa g$ is nondecreasing,
  suppose that $x\leq y$, and let $a=g\bigl(\frac{y+1}{2}\bigr)-g\bigl(\frac{x+1}{2}\bigr)$.
  Then we have
  \begin{align*}
    \Aa g(y) &\geq \frac{x+2}{3} g\Bigl(\frac{x+1}{2}\Bigr) g\Bigl(\frac{y+1}{2}\Bigr)
       + \frac{x+1}{3}g\Bigl(\frac{x}{2}\Bigr)\biggl(1-g\Bigl(\frac{y+1}{2}\Bigr)\biggr)\\
       &= \Aa g(x) + \biggl(\frac{x+2}{3}g\Bigl(\frac{x+1}{2}\Bigr) - 
         \frac{x+1}{3}g\Bigl(\frac{x}{2}\Bigr)\biggr)a \geq \Aa g(x).\qedhere
  \end{align*}
\end{proof}

We now analyze the behavior of $\Aa$ on the family of generating functions for Poisson random variables.
Recall that the generating function of $\Poi(a)$ is $e^{a(x-1)}$.
\begin{lemma}\thlabel{lem:a+ca}
   Define $g_a(x)=e^{a(x-1)}$ for all $a \geq 0$. For all $x \in [0,1]$,
    $$\Aa g_{a}(x) \leq g_{a + c_a}(x),$$
  where
  \begin{align}\label{eq:cadef}
    c_a = \begin{cases}
 \frac13 e^{-2}& 0 \leq a \leq 4,\\
 \frac13 e^{-a/2} & a \geq 4.
\end{cases}
  \end{align}  
\end{lemma}
\begin{proof}
  Applying the operator $\Aa$, we have
  \begin{equation}
    \begin{split}
      \Aa g_a(x) &=       \frac{x+2}{3}e^{a(x-1)} 
      +\frac{x+1}{3}e^{ax/2-a}\biggl(1
        -e^{a(x-1)/2}\biggr)\\
      &= g_a(x)r_{a/2}(x),
    \end{split}\label{eq:Aga}
  \end{equation}
  where 
  \begin{align*}
    r_b(x) &=\frac{2+x}{3}+\frac{1+x}{3}\left(e^{-bx}-e^{-b}\right).
  \end{align*}
  Note that $g_a(x)g_b(x)=g_{a+b}(x)$.
  It thus suffices to establish
  \begin{claim}
    For $x\in[0,1]$,
      we have
        $r_b(x) \leq g_{c_{2b}}(x)$.
  \end{claim}
  \begin{proof}[Proof of claim]
    We drop subscripts and let $r(x)=r_b(x)$ and $c=c_{2b}$.
Calculus and a little algebra show that
$$r'(x)=\frac{1}{3}\left(1-e^{-b}+e^{-bx}(-bx-b+1)\right)$$
and 
$$r''(x)=\frac{1}{3}e^{-bx}\left(b^2(x+1)-2b\right).$$
We break the proof up into cases. 
\begin{itemize}
\item If $b\leq 1$ then $r(x)$ is concave down on $[0,1]$ and the graph of $r(x)$ lies below its 
tangent line at $x=1$. Thus 
\begin{align*}
  r(x) \leq 1+r'(1)(x-1) &=1+\frac13\bigl[1-2be^{-b}\bigr](x-1)\\
    &\leq \exp\biggl[\frac13 \bigl(1-2be^{-b}\bigr)(x-1)\biggr].
\end{align*}
It is easily verified that $\frac13\bigl(1-2be^{-b}\bigr)\geq \frac13\bigl(1-2e^{-1}\bigr)
  \geq e^{-2}/3$ for $b\leq 1$ and hence that
$r(x)\leq g_c(x)$.
  
\item If $b\geq 2$ then $r(x)$ is concave up on $[0,1]$ and the graph of $r(x)$ lies below the 
secant line between $(0,r(0))$ and $(1,r(1))$. Thus as $r(1)=1$ we have
\begin{align*}
  r(x) \leq 1+(1-r(0))(x-1) &=1+\frac13e^{-b}(x-1)\\
   &\leq \exp\biggl[\frac13 e^{-b}(x-1)\biggr]=g_c(x).
\end{align*}

\item If $1<b<2$ then there is a unique inflection point at $I=\frac2b-1$ where $r$ switches from concave down to concave up.
  Since $r$ is concave up on $[I,1]$, the graph of $r$ lies 
  below the line connecting $(1,1)$ to $(I,r(I))$. Since $r$ is concave down on $[0,I]$, 
  to the left of $I$ the graph of $r$ 
  lies below its tangent line at $(I,r(I))$. Thus the line segment from $(I,r(I))$ to 
$(0,r(I)-r'(I)I)$ lies above $r$, as in Figure~\ref{fig:inflection}.
Therefore $r$ lies below the line between (1,1) and $(0,r(I)-r'(I)I)$, and
  \begin{align}
    r(x)\leq 1+(1-r(I)+Ir'(I))(x-1).\label{eq:rbound}
  \end{align}
  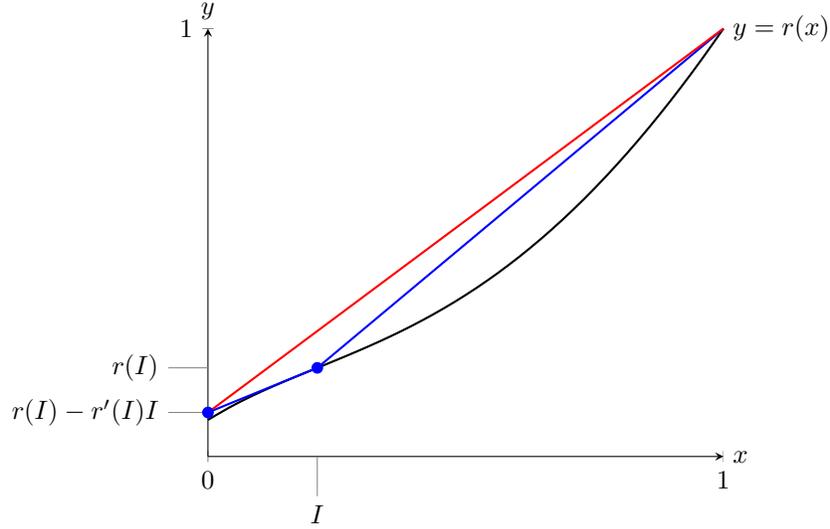
\begin{figure}
    \begin{center}
      \begin{tikzpicture}
        \begin{axis}[
          axis x line=bottom,
          axis y line=left,
          xtick={0, 1},
          ytick={1},
          extra x ticks={.2121},
          extra x tick labels={$I$},
          extra y ticks={.9372, .9445},
          extra y tick labels={$r(I)-r'(I)I$, $r(I)$},
          extra tick style={tickwidth=15pt, tick align=outside},
          xlabel=$x$,
          every axis x label/.append style={at=(current axis.right of origin),anchor=west},
          ylabel=$y$,
          every axis y label/.append style={at=(current axis.above origin), anchor=south,rotate=-90},
          domain=0:1,
          xmin=0, xmax=1,
          ymin=.93, ymax=1,
          smooth,
          clip=false
          ]
          \addplot[thick,mark=none] gnuplot{(2+x)/3 + (1+x) / 3 * (exp(-1.65 * x) - exp(-1.65))};
          \addplot[only marks,blue] coordinates{ (.2121, .9445) };
          \node[right] at (axis cs:1, 1) {$y=r(x)$};
          \draw[thick, blue] (axis cs: 1, 1)--(axis cs: .2121,.9445)--(axis cs: 0,.9372);
          \addplot[only marks, blue] coordinates{ (0, .9372) };
          \draw[thick, red] (axis cs: 1,1) -- (axis cs: 0, .9372);
        \end{axis}
      \end{tikzpicture}
      
    \end{center}
    \caption{Above the graph of $y=r(x)$ sits the secant line
    from $(I, r(I))$ to $(1,1)$ and the tangent line to $r(x)$ at $x=I$, both depicted in blue.
    Above them in red is the line $y=1+(1-r(I)+Ir'(I))(x-1)$.}\label{fig:inflection}
  \end{figure}

  Next, we evaluate 
  \begin{align} \label{eq:lhp}
    1-r(I)+Ir'(I) &= 1 - \frac13\biggl(2 + \Bigl(\frac{4}{b}-1\Bigr) e^{b-2} - e^{-b}\biggr)
  \end{align}
  and try to bound this expression from below for $b\in(1,2)$.
  We proceed as calculus students, looking for critical points in this interval.
  The derivative with respect to $b$ is
  \begin{align*}
    -\frac13\biggl(\Bigl(\frac{4}{b}-1-\frac{4}{b^2}\Bigr)e^{b-2} + e^{-b}\biggr),
  \end{align*}
  and a bit of algebra shows that the zeros of this expression are the solutions to
  \begin{align*}
    e^{2(b-1)} \biggl(\frac{2-b}{b}\biggr)^2 = 1.
  \end{align*}
  Taking logarithms, we are interested in solutions to
  \begin{align*}
    b-1 + \log(2-b) - \log b = 0.
  \end{align*}
  on $(1,2)$. On this interval we can replace the logarithms with their power series
  expansions around~$1$ to rewrite the left-hand side as
  \begin{align*}
    b - 1 + 2\biggl(\frac{(b-1)^2}{2} + \frac{(b-1)^4}{4} + \frac{(b-1)^6}{6}+\cdots\biggr),
  \end{align*}
  which is strictly positive for $b\in(1,2)$.
  Thus \eqref{eq:lhp} has no critical values on $(1,2)$, and its minimum on $[1,2]$
  is $e^{-2}/3$, occurring at $b=2$. Applying this to \eqref{eq:rbound}, we have shown that
  \begin{align*}
     r(x)\leq 1+\frac13 e^{-2}(x-1)  \leq \exp\biggl[\frac13 e^{-2}(x-1)\biggr]=g_c(x).
  \end{align*}  
\end{itemize}
This concludes the proof of both the claim and the lemma.
  \end{proof}\renewcommand{\qedsymbol}{}
\end{proof}
\begin{remark}
  Though the preceding lemma was an exercise in calculus, it has a probabilistic intepretation.
If we think of $\Aa$ as acting directly on distributions instead of on their generating functions,
this lemma shows that the result of applying $\Aa$ to $\Poi(a)$ is larger than $\Poi(a+c_a)$
in the probability generating function stochastic order described 
at the beginning of Section~\ref{sec:recurrence}. The reason that $\Aa g_a$ simplifies
so nicely in \eqref{eq:Aga} is the Poisson thinning property, and the fact that $g_a(x)g_b(x)=g_{a+b}(x)$
is just the statement that the sum of independent Poissons is Poisson. There is a temptation
to interpret $\Aa g_a(x) = g_a(x) r_{a/2}(x)$ as saying that the distribution resulting from
applying $\Aa$ to $\Poi(a)$ is a convolution of $\Poi(a)$ and another distribution, but $r_{a/2}(x)$
is not monotone in $x$ and hence not the generating function of a probability distribution.
\end{remark}

\begin{lemma}\thlabel{lem:limAa}
  For $x\in[0,1)$,
  \begin{align*}
    \lim_{n\to\infty} \Aa^n g_0(x) = 0.
  \end{align*}
\end{lemma}
\begin{proof}
  Define the sequence $a_n$ by $a_0=0$ and $a_{n+1} = a_n+c_{a_n}$.
  By \thref{sandinthevaseline,lem:preserveincreasing,lem:a+ca},
  \begin{align*}
    \Aa^n g_0(x) &\leq g_{a_{n}}(x) = e^{a_n(x-1)}.
  \end{align*}
  We need to show that $a_n\to\infty$ as $n\to\infty$.
  Suppose this does not hold.
  Since the sequence is increasing, $a_n\to a$ for some constant~$a$.
  Looking back at \eqref{eq:cadef}, this implies that $c_{a_n}$ converges to a strictly positive limit.
  We can then choose $n$ sufficiently large that $a_n + c_{a_n}>a$,
  a contradiction.
\end{proof}

\begin{proof}[Proof of \thref{thm:2tree}~(\ref{d2})]
Let $f$ be the generating function $f(x)= \E x^V$ with $V$ the number of visits to the root in the self-similar model frog model on the binary tree. By Proposition~\ref{mustangsally} we know that $f$ satisfies the recursion relation $\Aa f=f$. Since $f$ is a probability generating function, it satisfies
$f(x)\leq 1 =g_0(x)$ for $x \in[0, 1]$. 
Proposition~\ref{mustangsally} and \thref{sandinthevaseline,lem:preserveincreasing}
imply $f(x)\leq \Aa^ng_0(x)$ for all~$n$. By \thref{lem:limAa}, $f$ is identically zero on $[0,1)$. Thus the probability of any finite number of returns to the root is 0. This implies there are a.s.\ infinitely many returns to the root in the self-similar model. By the coupling in Proposition~\ref{kermit} each return in the self-similar model corresponds to a distinct return in the frog model. So, the frog model on the binary tree is a.s.\ recurrent.
\end{proof}

\section{Transience for $d \geq 5$} \label{sec:transience}
The non-backtracking model was useful in the previous section because it was dominated by the
usual frog model but was still recurrent. To prove transience, we instead seek processes that
dominate the frog model and can be proven transient. For example, consider a branching random walk on $\TT_d$ whose particles split in two at every step. 
Let $C_n$ be the $n$th Catalan number, which is the number of Dyck paths of length~$2n$. 
By a union bound, the probability that any of the $2^{2n}$ particles at time~$2n$ are at the root is at most
\[
2^{2n} \left( \f{ d}{d+1} \right)^n\left(\f 1 { d+1}  \right)^n C_n = O\left( \left(  16d/(d+1)^2 \right)^n \right).
\]
 When $d\geq 14$, this quantity is summable, and hence the branching random walk visits the root finitely many times. As this walk can be naturally coupled to the frog model so that every awake frog has a corresponding particle, this proves that the frog model is transient for $d\geq 14$.

In this section, we will present a series of refinements to this argument to ultimately prove \thref{thm:5tree}~(\ref{d5}).
In \thref{prop:6tree}, we use a branching random walk on the integers and martingale techniques to prove transience for $d\geq6$. 
We use this argument as a base for our proofs of \thref{prop:5.5tree}, transience 
on the deterministic tree which alternates between five and six children,
and \thref{thm:5tree}~(\ref{d5}), transience for $d\geq 5$.
Both proofs use a multitype branching random walk. 
We included \thref{prop:5.5tree} because its calculations can
easily be done by hand. In \thref{thm:5tree}~(\ref{d5}), on the other hand,
we use a branching random walk with 27 types. The necessary calculations
are intractable by hand, but they take only a few seconds on a computer. To get started we first address some difficulties that arise from reflection at the root. In doing so we also prove the 0-1 law described in \thref{thm:01law}.

\subsection{Couplings and 0-1 law}\label{sect:01}
We will need to consider the frog model on several modifications of a rooted tree. We can handle these special cases all at once by working in a more general setting. Let $\Tr$ be any infinite rooted graph and $H$ any graph. Enumerate finitely or countably many copies of $\Tr$ by $\Tr(i)$, and form a graph $G$ by adding an edge from the root of each $\Tr(i)$ into $H$. 
Our next lemma shows that regardless of the number of sleeping frogs placed on $H$, a frog model is less transient on $G$ than on $\Tr$.
Our motivation is the case when $\Lambda = \mathbb T_d$,
as in Corollaries~\ref{cor:homtree} and~\ref{cor:clearedtree}.
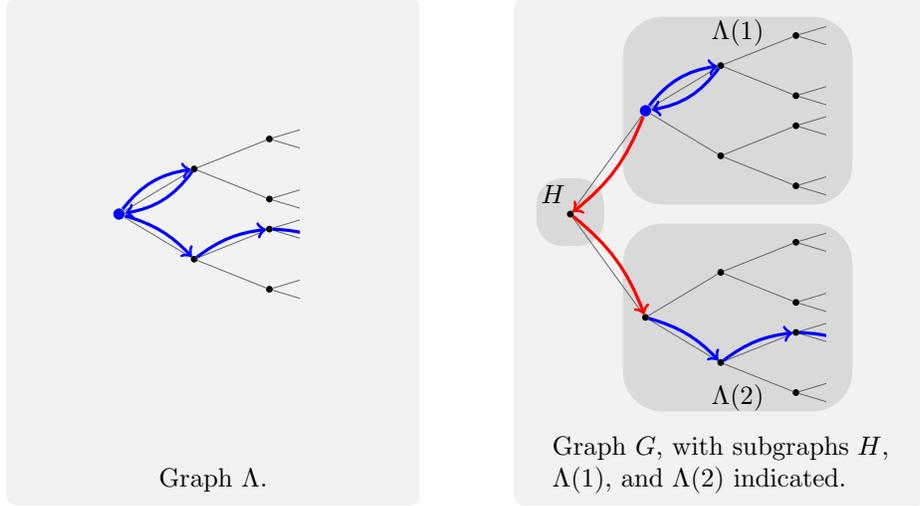
\begin{figure}
    \begin{center}
\begin{tikzpicture}[treelines/.style={very thin, black!60!white}, walklines/.style={very thick, blue,bend left=15},tv/.style={circle,fill,inner sep=0,
    minimum size=0.075cm,draw}, starttv/.style={circle,fill=blue,inner sep=0,
    minimum size=0.15cm}]
  \fill[black!05!white,rounded corners] (-1.5, -5.25) rectangle (4, 1.5);
      \node[anchor=base,align=left] at (1.25,-5) {Graph $\Lambda$.} ;
  \begin{scope}[shift={(0,-1.375)}]
  \path (0,0) node[starttv] (R) {}
        (1,-0.6) node[tv] (L1) {}
        (1, 0.6) node[tv] (L2) {}
        (2, 1) node[tv] (L22) {}
        (2, 0.2) node[tv] (L21) {}
        (2, -0.2) node[tv] (L12) {}
        (2,-1) node[tv] (L11) {}
        (L22)+(1,-0.3) coordinate (L221)
             +(1, 0.3) coordinate (L222)
        (L21)+(1,-0.3) coordinate (L211)
             +(1, 0.3) coordinate (L212)
        (L12)+(1,-0.3) coordinate (L121)
             +(1, 0.3) coordinate (L122)
        (L11)+(1,-0.3) coordinate (L111)
             +(1, 0.3) coordinate (L112);
  \clip (-1, -1.2) rectangle (2.4, 1.2);
  \draw[treelines] 
       (R)--(L1) (R)--(L2)
       (L1)--(L11) (L1)--(L12)
       (L2)--(L21) (L2)--(L22)
       (L22)--(L222) (L22)--(L221)
       (L21)--(L212) (L21)--(L211)
       (L12)--(L122) (L12)--(L121)
       (L11)--(L112) (L11)--(L111);
    \begin{scope}[walklines, ->]
      \draw[bend left=22] (R) to (L2);
      \draw[bend left=22] (L2) to (R);
      \draw (R) to (L1);
      \draw (L1) to (L12);
      \draw (L12) to (L121);
    \end{scope}
  \end{scope}

  \begin{scope}[shift={(7, 0)}]
    \fill[black!05!white,rounded corners] (-1.75, -5.25) rectangle (3.75, 1.5);
          \node[anchor=base,align=left] at (1,-5) {Graph $G$, with subgraphs $H$,\\
           $\Lambda(1)$, and $\Lambda(2)$ indicated.} ;
    \begin{scope}
      \fill[black!15!white,rounded corners=15pt] (-0.3, -1.25) rectangle (2.75, 1.25);
      \node[anchor=base] at (1.225,0.95) {$\Lambda(1)$} ;
      \path (0,0) node[starttv] (R) {}
        (1,-0.6) node[tv] (L1) {}
        (1, 0.6) node[tv] (L2) {}
        (2, 1) node[tv] (L22) {}
        (2, 0.2) node[tv] (L21) {}
        (2, -0.2) node[tv] (L12) {}
        (2,-1) node[tv] (L11) {}
        (L22)+(1,-0.3) coordinate (L221)
             +(1, 0.3) coordinate (L222)
        (L21)+(1,-0.3) coordinate (L211)
             +(1, 0.3) coordinate (L212)
        (L12)+(1,-0.3) coordinate (L121)
             +(1, 0.3) coordinate (L122)
        (L11)+(1,-0.3) coordinate (L111)
             +(1, 0.3) coordinate (L112);
      \clip (-1, -1.2) rectangle (2.4, 1.2);
      \draw[treelines] 
         (R)--(L1) (R)--(L2)
         (L1)--(L11) (L1)--(L12)
         (L2)--(L21) (L2)--(L22)
         (L22)--(L222) (L22)--(L221)
         (L21)--(L212) (L21)--(L211)
         (L12)--(L122) (L12)--(L121)
         (L11)--(L112) (L11)--(L111);
    \end{scope}

    \begin{scope}[shift={(0,-2.75)}]
      \fill[black!15!white,rounded corners=15pt] (-0.3, -1.25) rectangle (2.75, 1.25);
      \node[anchor=base] at (1.225,-1.15) {$\Lambda(2)$} ;
      \path (0,0) node[tv] (S) {}
        (1,-0.6) node[tv] (M1) {}
        (1, 0.6) node[tv] (M2) {}
        (2, 1) node[tv] (M22) {}
        (2, 0.2) node[tv] (M21) {}
        (2, -0.2) node[tv] (M12) {}
        (2,-1) node[tv] (M11) {}
        (M22)+(1,-0.3) coordinate (M221)
             +(1, 0.3) coordinate (M222)
        (M21)+(1,-0.3) coordinate (M211)
             +(1, 0.3) coordinate (M212)
        (M12)+(1,-0.3) coordinate (M121)
             +(1, 0.3) coordinate (M122)
        (M11)+(1,-0.3) coordinate (M111)
             +(1, 0.3) coordinate (M112);
      \clip (-1, -1.2) rectangle (2.4, 1.2);
      \draw[treelines] 
         (S)--(M1) (S)--(M2)
         (M1)--(M11) (M1)--(M12)
         (M2)--(M21) (M2)--(M22)
         (M22)--(M222) (M22)--(M221)
         (M21)--(M212) (M21)--(M211)
         (M12)--(M122) (M12)--(M121)
         (M11)--(M112) (M11)--(M111);
       \draw[walklines,->] (M12) to (M121);

    \end{scope}
    \fill[black!15!white,rounded corners=10pt] (-1.45, -1.8) rectangle (-0.55, -0.9);
    \node[anchor=base] at (-1.225,-1.225) {$H$} ;

    \path (-1, -1.375) node[tv] (negative) {};
    \draw[treelines] (R)--(negative) (S)--(negative);
    \begin{scope}[walklines, ->]
      \draw[bend left=22] (R) to (L2);
      \draw[bend left=22] (L2) to (R);
      \draw[red] (R) to (negative);
      \draw[red] (negative) to (S);
      \draw (S) to (M1);
      \draw (M1) to (M12);
    \end{scope}
    
  \end{scope}  
\end{tikzpicture}
    \end{center}
    \caption{The paths of frog~$x$ in $\Lambda$ and frog~$x'$ in $G$. Frog~$x$ follows the blue
    steps of $x'$ and ignores the red steps.}
    \label{fig:frogpathcoupling}
  \end{figure}
  
\begin{lemma}\thlabel{lem:Gcoupling}
  We consider two frog models. The first is on $\Tr$ with i.i.d.-$\nu$ initial conditions, for any
  measure $\nu$ on the nonnegative integers.
  The second is on $G$ with the following initial conditions: one initially active frog at the root
  of $\Tr(1)$; i.i.d.-$\nu$ sleeping frogs at all other vertices of $\bigcup_i \Tr(i)$; and
  any configuration of sleeping frogs in $H$.
  Assume $H$ is such that a random walk on $G$ a.s.\ escapes $H$. 
  
  Let $V_G$ be the number of times the root of any $\Tr(i)$ is visited in the frog model on $G$, 
  not counting steps from $H$ to a root. Let
  $V_{\Tr}$ be the the number of times the root is visited in the frog model on $\Tr$.
  Then the two frog models can
  be coupled so that $V_G \geq V_{\Tr}$.
\end{lemma}
\begin{proof}

Have the frog $x$, awake at the root $\emptyset \in \Tr$, mime the 
frog $x'$ that starts at the root of $\Tr(1)$. As depicted in Figure~\ref{fig:frogpathcoupling}, 
whenever $x'$ enters $H$ the frog $x$ pauses at $\emptyset$; when $x'$ re-enters  any $\Tr(i)$, the frog $x$ begins following $x'$ again. 

When $x$ visits a vertex that has yet to be visited, so will $x'$. Couple the number of sleeping frogs at
the vertices occupied by $x$ and $x'$, and couple the newly awoken frogs to each other as with
$x$ and $x'$. 
In this way, $x$ and all descendants on $\Tr$ perform simple random walks coupled to a frog on some $\Tr(i)$. 
Thus, every visit to the root in $\Tr$ corresponds to a visit to level~$0$ in $G$,
showing that $V_G\geq V_{\Tr}$ under this coupling.
\end{proof}

We give two corollaries. The first will help us prove our transience results,
and the second will help us prove a $0$-$1$ law for transience
and recurrence.
\begin{cor}\thlabel{cor:homtree}
  Consider the frog model on the $(d+1)$-homogeneous tree $\Thom$ starting
  with a single active frog at the root, and with no sleeping frog at direct
  ancestors of the root.   
  If level~$0$ is almost surely visited finitely many times in this model, then
  the frog model on $\TT_d$ is almost surely transient.
\end{cor}
\begin{proof}
  Let $G=\Thom$, thinking of it as countably
  many copies of $\TT_d$ each joined at its root to a leaf of the infinite graph consisting of all the
  negative levels of $G$. The statement then follows immediately from \thref{lem:Gcoupling}.
\end{proof}
\begin{cor}\thlabel{cor:clearedtree}
  Run the frog model on $\TT_d$, 
  starting with an active frog not at the root but at level~$k$.
  Assume that there are no sleeping frogs at levels $0,\ldots,k-1$ and i.i.d.-$\nu$ sleeping frogs
  at level~$k$ and beyond, with the exception of the location of the initial frog.
  The probability that the root is visited infinitely often in this
  model is at least the probability that the root is visited infinitely often in the usual
  frog model on $\TT_d$ with i.i.d.-$\nu$ initial conditions.
\end{cor}
\begin{proof}
  To set up our alternate frog model,
  let $G=\TT_d$, thinking of it as $d^k$~copies of $\TT_d$
  joined by a graph consisting of levels~$0$ to $k-1$ of the original graph.
  Let $p$ be the probability that the root is visited infinitely often in the usual frog model.  
  Let $Y$ be the number of visits from level~$k+1$ to $k$ in the alternate model.
  It follows immediately from \thref{lem:Gcoupling} that 
  $
    \P[Y=\infty]\geq p
  $.
  
  Let $X$ be the number of visits to the root.
  We would like to show that
  \begin{align}
    \P[Y=\infty,\,X<\infty]=0,\label{lem:lemresult}
  \end{align}
  thus proving that $\P[X=\infty]\geq p$.
  Call it a \emph{dash} if a frog  moves from level~$k+1$ to a vertex~$v$ at level~$k$, walks
  directly to the root, and then walks directly back to $v$.
  Let $X'$ be the total number of dashes that occur.
  Conditional on a frog stepping from level~$k+1$ to~$k$, it makes a dash independently
  of all other frogs, since
  the model has no sleeping frogs at levels~$0$ to~$k-1$.
  Whether or not it makes a dash is also independent of its own future number of visits
  from level~$k+1$ to~$k$ and of dashes. Thus,
  at every visit from level~$k+1$ to~$k$, there is an independent $1/d(d+1)^{2k-1}$ chance of a dash,
  showing that
  \begin{align*}
    \P[Y=\infty,\,X'<\infty]=0.
  \end{align*}
  Since $X'<X$, this shows
  \eqref{lem:lemresult} and completes the proof.
  \end{proof}

We are now ready to prove the 0-1 law.
\begin{proof}[Proof of \thref{thm:01law}]
  Suppose the probability that the root is visited infinitely often in
  the frog model on $\TT_d$ with i.i.d.-$\nu$ initial conditions is $p>0$. We wish to show that $p=1$.
  The idea of the proof is to turn this statement into a more finite event,
  and then show that there are infinitely many independent opportunities for this
  event to occur. To this end, fix a constant~$N$. We will show that at least $N$ frogs visit
  the root with probability~$1$.
  
  \begin{claim}
    For any $k$ and $N$, there is a constant $K=K(k,N)$ such that the following statement holds.
    Consider the frog model on $\TT_d$ starting with a frog at level~$k$,
    with i.i.d.-$\nu$ sleeping frogs at levels $k,k+1,\ldots,K-1$ with the exception
    of the vertex of the initial frog, and with no sleeping frogs outside of this range. 
    With probability at least $p/2$, this process
    makes at least $N$ visits to the root.
  \end{claim}
  \begin{proof}
    Consider the frog process with no sleeping frogs below level~$k$, as in \thref{cor:clearedtree}.
    Let $E_K$ be the event that there are at least $N$ visits to the root by frogs that are woken
    without the help of any frogs at level~$K$ or beyond. 
    As $K\to\infty$, the event $E_K$ converges
    upward to the event that there are at least $N$ visits to the root by any active frog, which
    occurs with at least probability~$p$ by \thref{cor:clearedtree}. Thus, for sufficiently
    large $K$, we have $\P[E_K] \geq p/2$.
  \end{proof}

  Now, we can find infinitely many independent events with probability~$p/2$, 
  each implying $N$ visits to the root. Let $k_0,k'_0=0$, and inductively choose $k_i,k'_i$ as
  follows. Let $k'_{i}=K(k_{i-1}, N)$
  from the claim. 
  Let $k_{i}$ be the level of the first frog that wakes up at level $k'_i$ or beyond (assuming that
  $\nu$ is not a point mass at $0$, there will be such a frog).
  Now, imagine a frog process starting with this frog, with no sleeping frogs below level~$k_i$
  or at level~$k_{i+1}$ or beyond. These processes can all be embedded into the original frog
  process on $\TT_d$, and each one independently has a $p/2$ chance of visiting the root at least
  $N$ times, by the claim. Thus, the root is visited at least $N$ times almost surely, for arbitrary $N$.
\end{proof}

\subsection{Proving transience} \label{sect:transience}
\tikzset{every picture/.style={grow=right}}
\tikzset{vert/.style={circle,fill,inner sep=0,
              minimum size=0.15cm,draw,outer sep=0}}

Consider the branching random walk where
each particle gives birth at each step either to one
child to its left or to two children to its right.
Formally, we define this as a sequence of point processes.
Start with $\pproc_0$ as a single particle at $0$.
With probability $1/(d+1)$, the point process $\pproc_1$ consists of a single
particle at $-1$; with probability $d/(d+1)$, it consists of two particles
at $1$. After this, each particle in $\pproc_n$ produces children in $\pproc_{n+1}$
in the same way relative to its position, independently of all other particles.
We will use this branching random walk to prove the frog model transient for $d\geq 6$
and closely related processes to extend this down to $d=5$.

\begin{prop}\thlabel{prop:6tree}
  For $d\geq 6$, the frog model on $\TT_d$ is almost surely transient.
\end{prop}

\begin{proof}
Consider the frog model on $\homtree_d$, starting with no sleeping frogs at direct
ancestors of the root, as
in \thref{cor:homtree}. When a frog jumps backward in this process,
it never spawns a new frog, and when it moves forward, it sometimes does.
Thus, the projection of this frog model onto the integers can be coupled
with $(\pproc_n,\,n\geq 0)$ so that every frog has a corresponding particle.
By \thref{cor:homtree}, proving that $\pproc_n$
visits $0$ finitely many times a.s.\ proves that the frog model on $\TT_d$ is transient a.s.

  To determine the behavior of $\pproc_n$, we define a weight function $w$ on point process configurations.
We refer to the position of a particle~$i$ in a point process
configuration by $P(i)$ and define
\begin{align}
  w(\pproc) = \sum_{i\in\pproc} e^{-\theta P(i)}, \label{eq:weightfunc}
\end{align}
with $\theta$ to be chosen later. Letting $\mu=\E w(\pproc_1)$ we have
\begin{align*}
  \E[w(\pproc_{n+1}) \mid \pproc_n] = \mu  w(\pproc_n),
\end{align*}
and so the sequence
$w(\pproc_n)/\mu^n$ is a martingale. As it is positive, it converges almost surely.
When $\mu< 1$ this means $w(\pproc_n) \to 0$. If a particle in $\pproc_n$ occupies the origin then $w(\pproc_n) \geq 1$, and so infinitely many visits to the origin prevents $w(\pproc_n)$ from converging.
Hence, $\mu<1$ implies the a.s.\ transience of $\pproc_n$. 
(In fact, this holds when $\mu=1$ as well, though we will not need this.)
It then suffices to show that there
exists $\theta$ making $\mu< 1$.
We compute
\begin{align*}
  \mu = \tfrac 1 { d+1} e^{\theta} +2 \tfrac d { d+1} e^{-\theta}.
\end{align*}
This is minimized by setting $\theta = \log(2d)/2$,
which makes $\mu=2\sqrt{2d}/(d+1)$. A bit of algebra shows that $\mu< 1$ when 
$d> 3+2\sqrt{2}\approx 5.83$.\end{proof}

By using a multitype branching process, we can extend this proof to show
transience for $\TT_5$. Before we do so, we
will show how it works in a setting where humans
can do the math without much assistance.
\begin{prop}\thlabel{prop:5.5tree}
    Let $\TT_{5,6}$ be the tree whose levels alternate between
  vertices with $5$ children and vertices with $6$ children, starting with the root
  having either $5$ or $6$ children. The frog model on this tree is transient a.s.
\end{prop}
\begin{proof}
Let $\homtree_{5,6}$ be the five-six children alternating homogeneous tree which contains $\TT_{5,6}$ and place a sleeping frog at each vertex except for direct ancestors of the root of $\TT_{5,6}$. \thref{lem:Gcoupling} implies that it suffices to prove transience of this frog model on $\homtree_{5,6}$.

First note that a frog at a vertex with five children has different probabilities of moving forwards or backwards than a frog at a vertex with six children. By design the tree deterministically alternates, so a frog also alternates between each state. 

When a frog moves backwards there is chance it immediately jumps forward to the same vertex,
which will never spawn a new frog. Similarly, when two frogs occupy the same site there is a chance both jump forward to the same vertex, spawning at most one frog, not two. The idea is to introduce additional particle types that act like frogs in these more advantageous states.

Consider a multitype branching random walk on $\ZZ$ with six particle types, $F_5$, $D_5$, $B_5$,
$F_6$, $D_6$, and $B_6$. The subscript accounts for whether a frog is at a vertex with 5 or 6 children. $B$ particles represent frogs that have just stepped
backward. $D$ particles represent two frogs at once, the waker and wakee at a vertex where a frog has just woken up. Last, $F$ particles represent single frogs with sleeping frogs present at all children. A visual depiction
of these particle types is provided in Figure~\ref{fig:particletypes}, and the distribution
of children for each particle type is defined in Figure~\ref{fig:5.5treechildren}. 

Let $\zeta_n$ be the branching random walk
in which particles reproduce independently with the given
child distributions. These distributions are chosen to match how the projections of frogs on the integers behave. Ignoring for a moment whether a frog is at a site with five or six children, when a frog jumps back it becomes of type $B$ and when a new frog wakes it and its waker consolidate into a type $D$ particle. Any extra frogs become a type $F$ particle. These particles then reproduce independently on a ``fresh" tree configured so that the particles always generate at least as many frogs as the projection of the actual frog model. For this reason we can couple the integer projection of the frog model on $\homtree_{5,6}$ with $\zeta_n$ so that the particles representing awake frogs are a subset of $\zeta_n$. It therefore suffices to prove that $\zeta_n$ is transient.
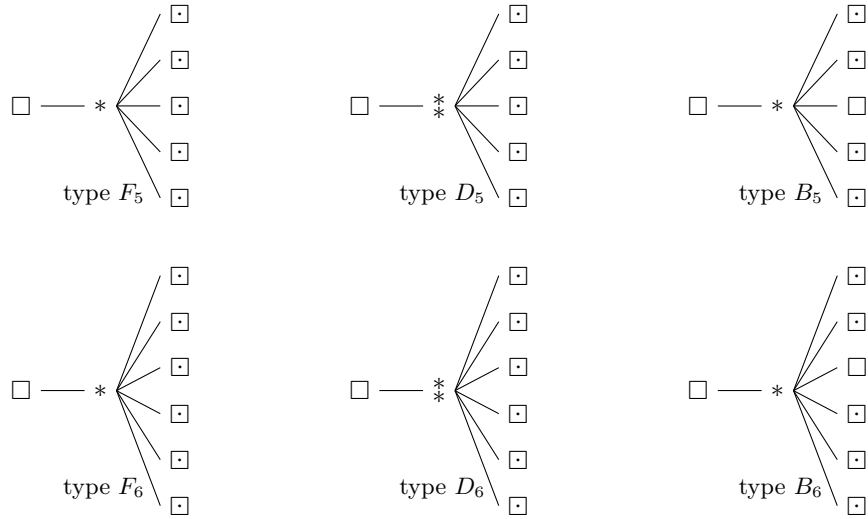
\begin{figure}
\begin{center}
  \begin{tikzpicture}[sibling distance=4pt, every level 2 node/.style={inner ysep=0}]
    \begin{scope}
      \Tree       [.$\square$ [.$*$ $\boxdot$  $\boxdot$ $\boxdot$ $\boxdot$ $\boxdot$    ]  ]  
      \node[font=\small] at (1.1,-1.175) {type $F_5$};
    \end{scope}
    
 \begin{scope}[shift={(4.5cm,0)}]
\Tree       [.$\square$ [.{$*$\\[-7pt]$*$} $\boxdot$  $\boxdot$ $\boxdot$ $\boxdot$ $\boxdot$    ]  ]  
      \node[font=\small] at (1.1,-1.175) {type $D_5$};
 \end{scope}
 \begin{scope}[shift={(9cm,0)}]
\Tree       [.$\square$ [.$*$ $\boxdot$  $\boxdot$ $\square$ $\boxdot$ $\boxdot$    ]  ]  
      \node[font=\small] at (1.1,-1.175) {type $B_5$};

 \end{scope}
 
\end{tikzpicture} \\[15pt]
\begin{tikzpicture}[sibling distance=4pt, every level 2 node/.style={inner ysep=0}]
 \begin{scope}
\Tree       [.$\square$ [.$*$ $\boxdot$ $\boxdot$  $\boxdot$ $\boxdot$ $\boxdot$ $\boxdot$    ]  ] 
      \node[font=\small] at (1.1,-1.3) {type $F_6$};
 \end{scope}
 
 \begin{scope}[shift={(4.5cm,0)}]
\Tree       [.$\square$ [.{$*$\\[-7pt]$*$} $\boxdot$ $\boxdot$  $\boxdot$ $\boxdot$ $\boxdot$ $\boxdot$    ]  ]
      \node[font=\small] at (1.1,-1.3) {type $D_6$};
 \end{scope}
 
 \begin{scope}[shift={(9cm,0)}]
\Tree       [.$\square$ [.$*$ $\boxdot$ $\boxdot$  $\boxdot$ $\square$ $\boxdot$ $\boxdot$    ]  ]  
      \node[font=\small] at (1.1,-1.3) {type $B_6$};
 \end{scope}
 
\end{tikzpicture}

\caption{A depiction of the six particle types from the proof of \thref{prop:5.5tree}.
Each asterisk is a frog represented by the particle. 
The symbol~$\boxdot$ signifies a vertex with a sleeping frog,
and the symbol~$\square$ represents a vertex with no sleeping frog.}
\label{fig:particletypes}
\end{center}
\end{figure}

\begin{figure}
  \begin{center}
  \begin{tikzpicture}[scale=0.55, particle/.style={draw,thick,shape=circle,inner sep=0,
                                                  anchor=south,font=\tiny,minimum size=1.225em}]
    \begin{scope}
      \begin{scope}
        \draw[very thick] 
        (0,0) arc[y radius=0.2, x radius=0.5, start angle=270, delta angle=35]
        (0,0) arc[y radius=0.2, x radius=0.5, start angle=270, delta angle=-35]
        (1,0) arc[y radius=0.2, x radius=0.5, start angle=270, delta angle=35]
        (1,0) arc[y radius=0.2, x radius=0.5, start angle=270, delta angle=-35]
        (-1,0) arc[y radius=0.2, x radius=0.5, start angle=270, delta angle=35]
        (-1,0) arc[y radius=0.2, x radius=0.5, start angle=270, delta angle=-35]
        (0,0) node[particle] {$\bm{F_5}$}
        (1.5,0) node (A) {};
      \end{scope}
      
      \begin{scope}[shift={(6,1)}]
        \draw[very thick] 
        (-1.5,0) node (B1) {}
        (0,0) arc[y radius=0.2, x radius=0.5, start angle=270, delta angle=35]
        (0,0) arc[y radius=0.2, x radius=0.5, start angle=270, delta angle=-35]
        (1,0) arc[y radius=0.2, x radius=0.5, start angle=270, delta angle=35]
        (1,0) arc[y radius=0.2, x radius=0.5, start angle=270, delta angle=-35]
        (-1,0) arc[y radius=0.2, x radius=0.5, start angle=270, delta angle=35]
        (-1,0) arc[y radius=0.2, x radius=0.5, start angle=270, delta angle=-35]
        (-1,0) node[particle] {$\bm{B_6}$};
      \end{scope}
      \begin{scope}[shift={(6,-1)}]
        \draw[very thick] 
        (-1.5,0) node (B2) {}
        (0,0) arc[y radius=0.2, x radius=0.5, start angle=270, delta angle=35]
        (0,0) arc[y radius=0.2, x radius=0.5, start angle=270, delta angle=-35]
        (1,0) arc[y radius=0.2, x radius=0.5, start angle=270, delta angle=35]
        (1,0) arc[y radius=0.2, x radius=0.5, start angle=270, delta angle=-35]
        (-1,0) arc[y radius=0.2, x radius=0.5, start angle=270, delta angle=35]
        (-1,0) arc[y radius=0.2, x radius=0.5, start angle=270, delta angle=-35]
        (1,0) node[particle] {$\bm{D_6}$};
      \end{scope}
      \draw[->] (A)--node[auto,font=\scriptsize,shift={(0.125,-0.075)}] {$\frac16$} (B1);
      \draw[->] (A)--node[auto,swap,font=\scriptsize,shift={(0.125,0.075)}] {$\frac56$}(B2);
    \end{scope}

        \begin{scope}[shift={(13,0)}]
      \begin{scope}
        \draw[very thick] 
        (0,0) arc[y radius=0.2, x radius=0.5, start angle=270, delta angle=35]
        (0,0) arc[y radius=0.2, x radius=0.5, start angle=270, delta angle=-35]
        (1,0) arc[y radius=0.2, x radius=0.5, start angle=270, delta angle=35]
        (1,0) arc[y radius=0.2, x radius=0.5, start angle=270, delta angle=-35]
        (-1,0) arc[y radius=0.2, x radius=0.5, start angle=270, delta angle=35]
        (-1,0) arc[y radius=0.2, x radius=0.5, start angle=270, delta angle=-35]
        (0,0) node[particle] {$\bm{F_6}$}
        (1.5,0) node (A) {};
      \end{scope}
      
      \begin{scope}[shift={(6,1)}]
        \draw[very thick] 
        (-1.5,0) node (B1) {}
        (0,0) arc[y radius=0.2, x radius=0.5, start angle=270, delta angle=35]
        (0,0) arc[y radius=0.2, x radius=0.5, start angle=270, delta angle=-35]
        (1,0) arc[y radius=0.2, x radius=0.5, start angle=270, delta angle=35]
        (1,0) arc[y radius=0.2, x radius=0.5, start angle=270, delta angle=-35]
        (-1,0) arc[y radius=0.2, x radius=0.5, start angle=270, delta angle=35]
        (-1,0) arc[y radius=0.2, x radius=0.5, start angle=270, delta angle=-35]
        (-1,0) node[particle] {$\bm{B_5}$};
      \end{scope}
      \begin{scope}[shift={(6,-1)}]
        \draw[very thick] 
        (-1.5,0) node (B2) {}
        (0,0) arc[y radius=0.2, x radius=0.5, start angle=270, delta angle=35]
        (0,0) arc[y radius=0.2, x radius=0.5, start angle=270, delta angle=-35]
        (1,0) arc[y radius=0.2, x radius=0.5, start angle=270, delta angle=35]
        (1,0) arc[y radius=0.2, x radius=0.5, start angle=270, delta angle=-35]
        (-1,0) arc[y radius=0.2, x radius=0.5, start angle=270, delta angle=35]
        (-1,0) arc[y radius=0.2, x radius=0.5, start angle=270, delta angle=-35]
        (1,0) node[particle] {$\bm{D_5}$};
      \end{scope}
      \draw[->] (A)--node[auto,font=\scriptsize,shift={(0.125,-0.075)}] {$\frac17$} (B1);
      \draw[->] (A)--node[auto,swap,font=\scriptsize,shift={(0.125,0.075)}] {$\frac67$}(B2);
    \end{scope}
    \begin{scope}[shift={(0,-5.5)}]
      \begin{scope}
        \draw[very thick] 
        (0,0) arc[y radius=0.2, x radius=0.5, start angle=270, delta angle=35]
        (0,0) arc[y radius=0.2, x radius=0.5, start angle=270, delta angle=-35]
        (1,0) arc[y radius=0.2, x radius=0.5, start angle=270, delta angle=35]
        (1,0) arc[y radius=0.2, x radius=0.5, start angle=270, delta angle=-35]
        (-1,0) arc[y radius=0.2, x radius=0.5, start angle=270, delta angle=35]
        (-1,0) arc[y radius=0.2, x radius=0.5, start angle=270, delta angle=-35]
        (0,0) node[particle] {$\bm{D_5}$}
        (1.5,0) node (A) {};
      \end{scope}
      
      \begin{scope}[shift={(6,2.6)}]
        \draw[very thick] 
        (-1.5,0) node (B1) {}
        (0,0) arc[y radius=0.2, x radius=0.5, start angle=270, delta angle=35]
        (0,0) arc[y radius=0.2, x radius=0.5, start angle=270, delta angle=-35]
        (1,0) arc[y radius=0.2, x radius=0.5, start angle=270, delta angle=35]
        (1,0) arc[y radius=0.2, x radius=0.5, start angle=270, delta angle=-35]
        (-1,0) arc[y radius=0.2, x radius=0.5, start angle=270, delta angle=35]
        (-1,0) arc[y radius=0.2, x radius=0.5, start angle=270, delta angle=-35]
        (-1,0) node[particle] {$\bm{D_6}$};
      \end{scope}
      \begin{scope}[shift={(6,1.25)}]
        \draw[very thick] 
        (-1.5,0) node (B2) {}
        (0,0) arc[y radius=0.2, x radius=0.5, start angle=270, delta angle=35]
        (0,0) arc[y radius=0.2, x radius=0.5, start angle=270, delta angle=-35]
        (1,0) arc[y radius=0.2, x radius=0.5, start angle=270, delta angle=35]
        (1,0) arc[y radius=0.2, x radius=0.5, start angle=270, delta angle=-35]
        (-1,0) arc[y radius=0.2, x radius=0.5, start angle=270, delta angle=35]
        (-1,0) arc[y radius=0.2, x radius=0.5, start angle=270, delta angle=-35]
        (-1,0) node[particle] {$\bm{B_6}$}
        (1,0) node[particle] {$\bm{D_6}$};
      \end{scope}
      \begin{scope}[shift={(6,-0.75)}]
        \draw[very thick] 
        (-1.5,0) node (B3) {}
        (0,0) arc[y radius=0.2, x radius=0.5, start angle=270, delta angle=35]
        (0,0) arc[y radius=0.2, x radius=0.5, start angle=270, delta angle=-35]
        (1,0) arc[y radius=0.2, x radius=0.5, start angle=270, delta angle=35]
        (1,0) arc[y radius=0.2, x radius=0.5, start angle=270, delta angle=-35]
        (-1,0) arc[y radius=0.2, x radius=0.5, start angle=270, delta angle=35]
        (-1,0) arc[y radius=0.2, x radius=0.5, start angle=270, delta angle=-35]
        (1,0) node[particle] (stacker) {$\bm{D_6}$}
        (stacker.north) node[particle] {$\bm{D_6}$};
      \end{scope}
      \begin{scope}[shift={(6,-2.6)}]
        \draw[very thick] 
        (-1.5,0) node (B4) {}
        (0,0) arc[y radius=0.2, x radius=0.5, start angle=270, delta angle=35]
        (0,0) arc[y radius=0.2, x radius=0.5, start angle=270, delta angle=-35]
        (1,0) arc[y radius=0.2, x radius=0.5, start angle=270, delta angle=35]
        (1,0) arc[y radius=0.2, x radius=0.5, start angle=270, delta angle=-35]
        (-1,0) arc[y radius=0.2, x radius=0.5, start angle=270, delta angle=35]
        (-1,0) arc[y radius=0.2, x radius=0.5, start angle=270, delta angle=-35]
        (1,0) node[particle] (stacker) {$\bm{D_6}$}
        (stacker.north) node[particle] {$\bm{F_6}$};
      \end{scope}
      \draw[->] (A)-- (B1)  node[pos=0.7,shift={(0,0.4)},font=\scriptsize] {$\frac{1}{36}$};
      \draw[->] (A)--(B2) node[pos=0.7,shift={(0,0.25)},font=\scriptsize] {$\frac{5}{18}$};
      \draw[->] (A)--(B3) node[pos=0.7,shift={(0,0.2)},font=\scriptsize] {$\frac{5}{9}$};
      \draw[->] (A)--(B4) node[pos=0.7,shift={(0,-0.375)},font=\scriptsize] {$\frac{5}{36}$};
    \end{scope}

    \begin{scope}[shift={(13,-5.5)}]
      \begin{scope}
        \draw[very thick] 
        (0,0) arc[y radius=0.2, x radius=0.5, start angle=270, delta angle=35]
        (0,0) arc[y radius=0.2, x radius=0.5, start angle=270, delta angle=-35]
        (1,0) arc[y radius=0.2, x radius=0.5, start angle=270, delta angle=35]
        (1,0) arc[y radius=0.2, x radius=0.5, start angle=270, delta angle=-35]
        (-1,0) arc[y radius=0.2, x radius=0.5, start angle=270, delta angle=35]
        (-1,0) arc[y radius=0.2, x radius=0.5, start angle=270, delta angle=-35]
        (0,0) node[particle] {$\bm{D_6}$}
        (1.5,0) node (A) {};
      \end{scope}
      
      \begin{scope}[shift={(6,2.6)}]
        \draw[very thick] 
        (-1.5,0) node (B1) {}
        (0,0) arc[y radius=0.2, x radius=0.5, start angle=270, delta angle=35]
        (0,0) arc[y radius=0.2, x radius=0.5, start angle=270, delta angle=-35]
        (1,0) arc[y radius=0.2, x radius=0.5, start angle=270, delta angle=35]
        (1,0) arc[y radius=0.2, x radius=0.5, start angle=270, delta angle=-35]
        (-1,0) arc[y radius=0.2, x radius=0.5, start angle=270, delta angle=35]
        (-1,0) arc[y radius=0.2, x radius=0.5, start angle=270, delta angle=-35]
        (-1,0) node[particle] {$\bm{D_5}$};
      \end{scope}
      \begin{scope}[shift={(6,1.25)}]
        \draw[very thick] 
        (-1.5,0) node (B2) {}
        (0,0) arc[y radius=0.2, x radius=0.5, start angle=270, delta angle=35]
        (0,0) arc[y radius=0.2, x radius=0.5, start angle=270, delta angle=-35]
        (1,0) arc[y radius=0.2, x radius=0.5, start angle=270, delta angle=35]
        (1,0) arc[y radius=0.2, x radius=0.5, start angle=270, delta angle=-35]
        (-1,0) arc[y radius=0.2, x radius=0.5, start angle=270, delta angle=35]
        (-1,0) arc[y radius=0.2, x radius=0.5, start angle=270, delta angle=-35]
        (-1,0) node[particle] {$\bm{B_5}$}
        (1,0) node[particle] {$\bm{D_5}$};
      \end{scope}
      \begin{scope}[shift={(6,-0.75)}]
        \draw[very thick] 
        (-1.5,0) node (B3) {}
        (0,0) arc[y radius=0.2, x radius=0.5, start angle=270, delta angle=35]
        (0,0) arc[y radius=0.2, x radius=0.5, start angle=270, delta angle=-35]
        (1,0) arc[y radius=0.2, x radius=0.5, start angle=270, delta angle=35]
        (1,0) arc[y radius=0.2, x radius=0.5, start angle=270, delta angle=-35]
        (-1,0) arc[y radius=0.2, x radius=0.5, start angle=270, delta angle=35]
        (-1,0) arc[y radius=0.2, x radius=0.5, start angle=270, delta angle=-35]
        (1,0) node[particle] (stacker) {$\bm{D_5}$}
        (stacker.north) node[particle] {$\bm{D_5}$};
      \end{scope}
      \begin{scope}[shift={(6,-2.6)}]
        \draw[very thick] 
        (-1.5,0) node (B4) {}
        (0,0) arc[y radius=0.2, x radius=0.5, start angle=270, delta angle=35]
        (0,0) arc[y radius=0.2, x radius=0.5, start angle=270, delta angle=-35]
        (1,0) arc[y radius=0.2, x radius=0.5, start angle=270, delta angle=35]
        (1,0) arc[y radius=0.2, x radius=0.5, start angle=270, delta angle=-35]
        (-1,0) arc[y radius=0.2, x radius=0.5, start angle=270, delta angle=35]
        (-1,0) arc[y radius=0.2, x radius=0.5, start angle=270, delta angle=-35]
        (1,0) node[particle] (stacker) {$\bm{D_5}$}
        (stacker.north) node[particle] {$\bm{F_5}$};
      \end{scope}
      \draw[->] (A)-- (B1)  node[pos=0.7,shift={(0,0.4)},font=\scriptsize] {$\frac{1}{49}$};
      \draw[->] (A)--(B2) node[pos=0.7,shift={(0,0.25)},font=\scriptsize] {$\frac{12}{49}$};
      \draw[->] (A)--(B3) node[pos=0.7,shift={(0,0.225)},font=\scriptsize] {$\frac{30}{49}$};
      \draw[->] (A)--(B4) node[pos=0.7,shift={(0,-0.375)},font=\scriptsize] {$\frac{6}{49}$};
    \end{scope}
    \begin{scope}[shift={(0,-11)}]
      \begin{scope}
        \draw[very thick] 
        (0,0) arc[y radius=0.2, x radius=0.5, start angle=270, delta angle=35]
        (0,0) arc[y radius=0.2, x radius=0.5, start angle=270, delta angle=-35]
        (1,0) arc[y radius=0.2, x radius=0.5, start angle=270, delta angle=35]
        (1,0) arc[y radius=0.2, x radius=0.5, start angle=270, delta angle=-35]
        (-1,0) arc[y radius=0.2, x radius=0.5, start angle=270, delta angle=35]
        (-1,0) arc[y radius=0.2, x radius=0.5, start angle=270, delta angle=-35]
        (0,0) node[particle] {$\bm{B_5}$}
        (1.5,0) node (A) {};
      \end{scope}
      
      \begin{scope}[shift={(6,1.25)}]
        \draw[very thick] 
        (-1.5,0) node (B1) {}
        (0,0) arc[y radius=0.2, x radius=0.5, start angle=270, delta angle=35]
        (0,0) arc[y radius=0.2, x radius=0.5, start angle=270, delta angle=-35]
        (1,0) arc[y radius=0.2, x radius=0.5, start angle=270, delta angle=35]
        (1,0) arc[y radius=0.2, x radius=0.5, start angle=270, delta angle=-35]
        (-1,0) arc[y radius=0.2, x radius=0.5, start angle=270, delta angle=35]
        (-1,0) arc[y radius=0.2, x radius=0.5, start angle=270, delta angle=-35]
        (-1,0) node[particle] {$\bm{B_6}$};
      \end{scope}
      \begin{scope}[shift={(6,0)}]
        \draw[very thick] 
        (-1.5,0) node (B2) {}
        (0,0) arc[y radius=0.2, x radius=0.5, start angle=270, delta angle=35]
        (0,0) arc[y radius=0.2, x radius=0.5, start angle=270, delta angle=-35]
        (1,0) arc[y radius=0.2, x radius=0.5, start angle=270, delta angle=35]
        (1,0) arc[y radius=0.2, x radius=0.5, start angle=270, delta angle=-35]
        (-1,0) arc[y radius=0.2, x radius=0.5, start angle=270, delta angle=35]
        (-1,0) arc[y radius=0.2, x radius=0.5, start angle=270, delta angle=-35]
        (1,0) node[particle] {$\bm{F_6}$};
      \end{scope}
      \begin{scope}[shift={(6,-1.25)}]
        \draw[very thick] 
        (-1.5,0) node (B3) {}
        (0,0) arc[y radius=0.2, x radius=0.5, start angle=270, delta angle=35]
        (0,0) arc[y radius=0.2, x radius=0.5, start angle=270, delta angle=-35]
        (1,0) arc[y radius=0.2, x radius=0.5, start angle=270, delta angle=35]
        (1,0) arc[y radius=0.2, x radius=0.5, start angle=270, delta angle=-35]
        (-1,0) arc[y radius=0.2, x radius=0.5, start angle=270, delta angle=35]
        (-1,0) arc[y radius=0.2, x radius=0.5, start angle=270, delta angle=-35]
        (1,0) node[particle] {$\bm{D_6}$};
      \end{scope}
      \draw[->] (A)-- (B1)  node[pos=0.7,shift={(0,0.3)},font=\scriptsize] {$\frac{1}{6}$};
      \draw[->] (A)--(B2) node[pos=0.7,shift={(0,0.2)},font=\scriptsize] {$\frac16$};
      \draw[->] (A)--(B3) node[pos=0.7,shift={(0,-0.225)},font=\scriptsize] {$\frac23$};
    \end{scope}

    \begin{scope}[shift={(13,-11)}]
      \begin{scope}
        \draw[very thick] 
        (0,0) arc[y radius=0.2, x radius=0.5, start angle=270, delta angle=35]
        (0,0) arc[y radius=0.2, x radius=0.5, start angle=270, delta angle=-35]
        (1,0) arc[y radius=0.2, x radius=0.5, start angle=270, delta angle=35]
        (1,0) arc[y radius=0.2, x radius=0.5, start angle=270, delta angle=-35]
        (-1,0) arc[y radius=0.2, x radius=0.5, start angle=270, delta angle=35]
        (-1,0) arc[y radius=0.2, x radius=0.5, start angle=270, delta angle=-35]
        (0,0) node[particle] {$\bm{B_6}$}
        (1.5,0) node (A) {};
      \end{scope}
      
      \begin{scope}[shift={(6,1.25)}]
        \draw[very thick] 
        (-1.5,0) node (B1) {}
        (0,0) arc[y radius=0.2, x radius=0.5, start angle=270, delta angle=35]
        (0,0) arc[y radius=0.2, x radius=0.5, start angle=270, delta angle=-35]
        (1,0) arc[y radius=0.2, x radius=0.5, start angle=270, delta angle=35]
        (1,0) arc[y radius=0.2, x radius=0.5, start angle=270, delta angle=-35]
        (-1,0) arc[y radius=0.2, x radius=0.5, start angle=270, delta angle=35]
        (-1,0) arc[y radius=0.2, x radius=0.5, start angle=270, delta angle=-35]
        (-1,0) node[particle] {$\bm{B_5}$};
      \end{scope}
      \begin{scope}[shift={(6,0)}]
        \draw[very thick] 
        (-1.5,0) node (B2) {}
        (0,0) arc[y radius=0.2, x radius=0.5, start angle=270, delta angle=35]
        (0,0) arc[y radius=0.2, x radius=0.5, start angle=270, delta angle=-35]
        (1,0) arc[y radius=0.2, x radius=0.5, start angle=270, delta angle=35]
        (1,0) arc[y radius=0.2, x radius=0.5, start angle=270, delta angle=-35]
        (-1,0) arc[y radius=0.2, x radius=0.5, start angle=270, delta angle=35]
        (-1,0) arc[y radius=0.2, x radius=0.5, start angle=270, delta angle=-35]
        (1,0) node[particle] {$\bm{F_5}$};
      \end{scope}
      \begin{scope}[shift={(6,-1.25)}]
        \draw[very thick] 
        (-1.5,0) node (B3) {}
        (0,0) arc[y radius=0.2, x radius=0.5, start angle=270, delta angle=35]
        (0,0) arc[y radius=0.2, x radius=0.5, start angle=270, delta angle=-35]
        (1,0) arc[y radius=0.2, x radius=0.5, start angle=270, delta angle=35]
        (1,0) arc[y radius=0.2, x radius=0.5, start angle=270, delta angle=-35]
        (-1,0) arc[y radius=0.2, x radius=0.5, start angle=270, delta angle=35]
        (-1,0) arc[y radius=0.2, x radius=0.5, start angle=270, delta angle=-35]
        (1,0) node[particle] {$\bm{D_5}$};
      \end{scope}
      \draw[->] (A)-- (B1)  node[pos=0.7,shift={(0,0.3)},font=\scriptsize] {$\frac{1}{7}$};
      \draw[->] (A)--(B2) node[pos=0.7,shift={(0,0.2)},font=\scriptsize] {$\frac17$};
      \draw[->] (A)--(B3) node[pos=0.7,shift={(0,-0.225)},font=\scriptsize] {$\frac57$};
    \end{scope}
    
  \end{tikzpicture}
  \end{center}
  
\caption{The distribution of children for each particle type in the proof of
         \thref{prop:5.5tree}.}
\label{fig:5.5treechildren}
\end{figure}
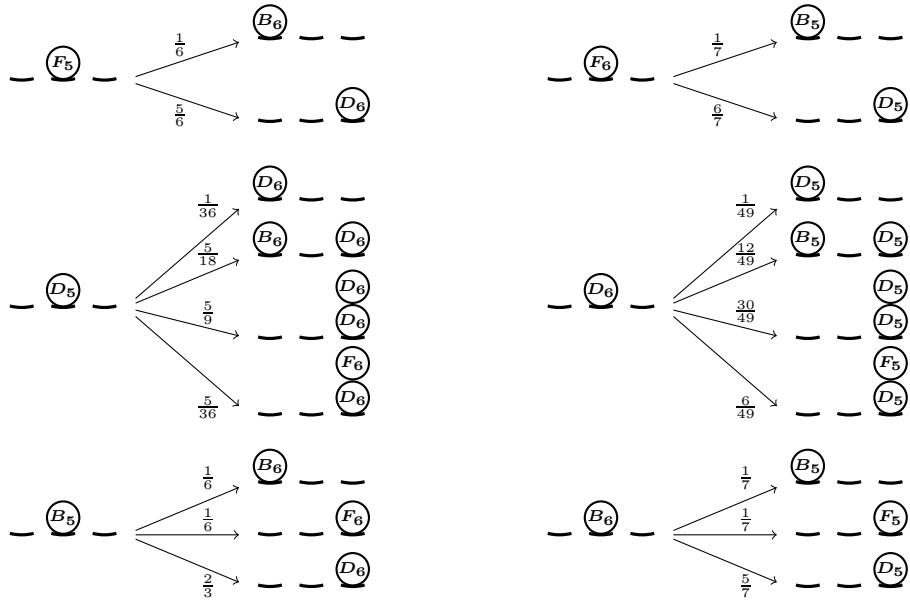

To analyze $\zeta_n$, we use a generalization of the martingale from \thref{prop:6tree}
to the multitype setting, introduced in \cite{first_and_last}.
Let $\zeta_n = \sum_i \zeta_n^{i}$, where $i$ ranges over the six particle types
and $\zeta_n^{i}$ denotes the restriction of $\zeta_n$ to particles of type~$i$.
Recalling the weight function $w$ given by \eqref{eq:weightfunc}, we define a matrix
$\Phi(\theta)$ by
\begin{align*}
  \Phi_{ij}(\theta) &= \E_i\Bigl[ w\bigl(\zeta_1^{j}\bigr)\Bigr].
\end{align*}
Here, we use $\E_i$ to denote expectation when $\zeta_0$ is a single particle at the origin
of type~$i$. Let $w_n$ denote a row vector whose $i$th entry is $w\bigl(\zeta_n^{i}\bigr)$.
Then
\begin{align}
  \E[w_{n+1} \mid  \zeta
  _n]
    &= w_n\Phi(\theta). \label{eq:martingaleeq}
\end{align}
Thus, for any eigenvalue~$\lambda$ and associated right eigenvector~$v$ of $\Phi(\theta)$,
\begin{align*}
  \E[w_{n+1}v \mid \zeta_n] &= w_n\Phi(\theta)v = \lambda w_nv,
\end{align*}
and so $w_nv / \lambda^n$ is a martingale. 

Since $\Phi(\theta)$ is a nonnegative irreducible matrix,
there is a positive eigenvalue
$\phi(\theta)$ equal to the spectral radius of $\Phi(\theta)$ by the Perron--Frobenius theorem. The eigenvector $v(\theta)$
associated with
$\phi(\theta)$ has strictly positive entries. We then have a positive martingale
$w_nv(\theta) / \phi(\theta)^n$. If $\phi(\theta)< 1$, then it follows
as in \thref{prop:6tree} that
the branching random walk visits $0$ finitely often,
thus proving that the frog model is almost surely transient.

All that remains is to find some value of $\theta$ such that $\phi(\theta)< 1$.
Ordering the rows and columns $F_5$, $D_5$, $B_5$, $F_6$, $D_6$, $B_6$ and reading
off $\E_i\bigl[w(\zeta_i^j)\bigr]$ from Figure~\ref{fig:5.5treechildren},
\begin{align*}
  \Phi(\theta) &=
  \begin{bmatrix}
    0 & 0 & 0 & 0 & \frac56 e^{-\theta} & \frac16 e^{\theta}\\[4pt]
    0 & 0 & 0 & \frac{5}{36} e^{-\theta} & \frac{1}{36}e^\theta + \frac{55}{36}e^{-\theta}
        & \frac{5}{18}e^{\theta}\\[4pt]
    0 & 0 & 0 & \frac16 e^{-\theta} & \frac23 e^{-\theta} & \frac16 e^{\theta}\\[4pt]
    0 & \frac67 e^{-\theta} & \frac17 e^{\theta} &0 & 0 & 0\\[4pt]
    \frac{6}{49} e^{-\theta} & \frac{1}{49}e^\theta + \frac{78}{49}e^{-\theta}
        & \frac{12}{49}e^{\theta}&0 & 0 & 0\\[4pt]
    \frac17 e^{-\theta} & \frac57 e^{-\theta} & \frac17 e^{\theta}&0 & 0 & 0\\[4pt]
  \end{bmatrix}.
\end{align*}
Computing the eigenvalues of this matrix numerically, one can confirm
that there exists $\theta$ with $\phi(\theta)< 1$; for example,
 $\phi(\log 3)\approx 0.9937$.
To be completely certain that this is not an artifact of rounding,
we will justify that $\phi(\log 3)<1$ without using floating-point arithmetic.
Observe that $\Phi(\log 3)$ has rational entries. Using the computer algebra system
SAGE, we calculated $(\Phi(\log 3))^{66}$ using exact arithmetic, and we found that its
largest row sum was less than $1$. (The only significance of the $66$th power is that
it is the lowest one for which this is true.)
This implies that all eigenvalues of $(\Phi(\log 3))^{66}$ are less than 1, which implies
that all eigenvalues of $\Phi(\log 3)$ are less than 1 as well.
The source code accompanying this paper includes this matrix and has instructions
so that readers can easily check these claims.
\end{proof}

  \begin{remark} We chose to include this proof to illustrate the technique we use to prove \thref{thm:5tree}~(\ref{d5}). Furthermore, this provides an example of proving the frog model is transient on an interpolation between different degree trees. This is relevant because the sharpest proof of \thref{conj:phase_transition} would find exactly where the phase transition occurs between recurrence and transience on $\TT_d$, 
 perhaps between $d=3$ and $d=4$. Last, a natural generalization is a frog model on Galton-Watson trees. Our argument depends on the deterministic structure of $\TT_{5,6}$ and we do not see an obvious way to generalize it.  

 \end{remark}
 

Having proven transience for the frog model on $\TT_d$ with $d\geq 6$ and on $\TT_{5,6}$, 
we present our final refinement to prove the  $\TT_5$ case.
  The proof is essentially the same as the previous one, but with more particle
  types and a more difficult calculation.

\begin{proof}[Proof of \thref{thm:5tree}~(\ref{d5})]
   
  We define a particle type $P(a, b, c)$, for $a\geq 1$ and $b,c\geq 0$.
  A particle of type $P(a,b,c)$ represents
  $a$ frogs on one vertex. 
  There are no sleeping frogs on at least $b$ of the vertex's children
  and on at least $c$ of the vertex's siblings. In this scheme, the $F$~types from the
  previous proof would translate
  to $P(1,0,0)$, the $D$~types would translate to $P(2,0,0)$,
  and the $B$~types would translate to $P(1,1,0)$.
  
  We use 27 of these particles, $P(a,b,c)$ with $1\leq a\leq 3$
  and $0\leq b,c\leq 2$.
  For particle type $P(a,b,c)$, consider the frog model on the homogeneous
  tree, starting with $a$ frogs at the root. As usual, remove the sleeping frogs
  from direct ancestors
  of the root. Also remove the sleeping frogs from $b$ children of the root
  and from $c$ siblings. From each of these 27 initial states, we compute all possible
  states to which the frog model could transition in two steps, along with the
  exact probabilities of doing so. We then represent each of these final states
  as a collection of particles of the 27~types, at levels $-2$, $0$, and
  $2$ on the tree. In this way, we determine child distributions for each
  particle type, as in Figure~\ref{fig:5.5treechildren}.
  There is a slight ambiguity in how to do this, as a state of frogs can be represented
  in more than one way by these particle types. For example, four frogs on one vertex with
  one sibling vertex with no sleeping frog could be represented as two particles
  of type $P(2,0,1)$, or as one of type $P(3,0,1)$ and one of type $P(1,0,1)$.
  We always chose particles greedily, opting for as many $3$-frog particles as possible. 
  Whatever choice we make here, our branching random walk will
  still dominate the frog model, since when we assign new particles we ``reset'' the tree below them so the particles wake at least as many frogs as their counterpart in the frog model.
  
  As in \thref{prop:5.5tree}, it suffices to compute the matrix $\Phi(\theta)$ and
  show that for some choice of $\theta$, its highest eigenvalue is less than one. 
  Our attached source code computes $\Phi(\theta)$ exactly.
  We include additional documentation there explaining how we performed this calculation
  and describing the steps we took to make sure it was trustworthy.
  To avoid rounding issues, we proceeded as with \thref{prop:5.5tree}.
  We exactly computed $(\Phi(\log 3))^{1024}$ by succesively squaring the matrix ten times,
  and we then checked that all of its row sums were less than 1.
  (There is no significance to the value $\log 3$; it just happens to work.)
  Thus, all eigenvalues of $(\Phi(\log 3))^{1024}$ are less than 1, implying
  that all eigenvalues of $\Phi(\log 3)$ are less than 1 as well.
  \end{proof}



\section{A frog model without a 0-1 law}

We obtain a graph on which the frog model does not satisfy a 0-1 law by combining the transient graph $\TT_6$ with the recurrent graph $\Z$ and proving that there is a positive probability that the frogs in each do not interact much.

To this end, we first prove two lemmas. The first shows there is a positive probability that the rightmost frog on the $\Z$ part of the combined graph
 escapes to $\infty$ while avoiding 0. This is necessary to rule out the possibility that too many frogs from $\Z$ get lost forever in $\TT_6$. The second lemma proves there is a positive probability a frog model on $\TT_6$ never returns to the origin, thus establishing a chance that the frog model on $G$ gets lost in the transience of $\TT_6$. 

\begin{lemma} \thlabel{lem:right}
  Consider the frog model on $G$, the graph formed by merging the root of $\TT_6$ and the origin
  of $\ZZ$. With positive probability, the frogs starting at $1,2,\ldots$ in $\ZZ$ all wake up.
\end{lemma}

\begin{proof}
  Let 
  \begin{align*}
    \delta_n = 
    \frac18 \prod_{k=1}^{n-1} \biggl(1-\frac{1}{(k+1)^2}\biggr),
  \end{align*}
  taking $\delta_1=1/8$.
  We will show by induction that the frogs at $1,\ldots,n$ wake up with probability
  at least $\delta_n$.
  When $n=1$, this holds because the initial frog moves right on its first step with probability~$1/8$.
  Now, assume the statement for $n$. Condition on the frogs at $1,\ldots,n$ being woken.
  From the time when the frog at $n$ is woken on, the two frogs there are independent random walkers,
  and at least one of them reaches $n+1$ before $0$ with probability $1-1/(n+1)^2$ by a standard
  martingale argument. Thus frog $n+1$ is woken with probability at least 
  $\delta_n\bigl(1-1/(n+1)^2\bigr)=\delta_{n+1}$, completing the induction.
  
  Taking a limit of increasing events, the probability of the frogs at $1,2,\ldots$ all waking
  is at least $\lim_{n\to\infty}\delta_n>0$.
\end{proof}

\begin{lemma} \thlabel{lem:lost}
Let $p'$ be the probability that the root is never visited past the initial frog's first move in the frog
model on $\TT_6$. It holds that $p' >0 $. 

\end{lemma}

\begin{proof}
  As in \thref{prop:6tree}, consider the frog model on $\TT^{\text{hom}}_6$, starting
  with no sleeping frogs at direct ancestors of the root.
  By \thref{lem:Gcoupling}, following the reasoning of \thref{cor:homtree},
  there is a coupling
  so that the number of visits to level~$0$ in the frog model on $\TT^{\text{hom}}_6$
  is at least the number of visits to the root in the frog model on $\TT_6$.

  Now, recall from \thref{prop:6tree} the point process~$\xi$, a branching random walk on $\ZZ$
  in which particles split whenever they move in the positive direction. This process dominates
  the projection of the frog model on $\TT^{\text{hom}}_6$ onto the integers. Putting this all together,
  it suffices to show that with positive probability, $\xi_n$ avoids $0$ for all $n\geq 1$.
  
  Suppose not, so $\xi$ a.s.\ revisits $0$. Since particles in $\xi$ reproduce independently,
  this implies that $\xi$ returns to the origin infinitely often. This is a contradiction, as we showed the opposite
  in proving \thref{prop:6tree}.
\end{proof}

\begin{proof}[Proof of \thref{thm:counterexample}]

In two steps, we bound the probability~$p$ of recurrence 
on $G=(\ZZ\cup \TT_6) / \{0 \sim \emptyset\}$:

\begin{description}

    \item[($p>0$)] The probability is $0$ that any frog starting in $\ZZ$ wakes but fails to visit $0$,
      by the recurrence of simple random walk on $\ZZ$.
      All frogs at $1,2,\ldots$ wake up with positive probability by \thref{lem:right}, and on this
      event they therefore all visit $0$.

    \item[($p<1$)] With probability $6/8$ the first jump of the frog at $0 \in G$ will be into $\TT_6$. Conditional on this, \thref{lem:lost} guarantees a frog model in this configuration will never again visit the origin with probability $p'>0$. Therefore, $1-p \geq \f 68 p'$. \qedhere
\end{description}
\end{proof}


\section{Conjectures}
\label{sec:conjectures}

Simulations suggest that for $d=3$ the frog model on $\mathbb T_d$ is recurrent a.s., while for $d = 4$ the model is transient a.s.
Our approach was to consider the frog model with the addition of stunning fences at each depth. When a frog jumps on a fence for the first time, it is stunned and stops moving.
When all frogs are stunned at depth~$k$, the fence turns off, and frogs resume their motion until
they reach depth~$k+1$ and are stunned again. 
Let $A_{d,k}$ be the number of stunned frogs 
that pile up on the fence at depth $k$ before it turns off.
We then examined the growth of $A_{d,k}$ in $k$ for different choices of $d$.
(The more obvious approach of directly simulating the frog model and counting visits to the root
does not yield any obvious conclusions, as the rapid growth of the frog model makes it impossible
to simulate very far.)


Martingale techniques tell us that the probability a frog at distance $k$ from the root reaches the root before visiting depth $k+1$ is greater than $c d^{-k}$ for some $c>0$ independent of $k$. It follows that  $$\E[\text{visits to root between $k$th and $(k+1)$th stunnings}] \geq c d^{-k} \E[A_{d,k}].$$ So, if $\sum_{k=1}^\infty d^{-k}\E[ A_{d,k} ] = \infty$, then the expected number of visits to the root is infinite, which strongly suggests the model is
recurrent.

This occurs if $kd^{-k}\E[A_{d,k}]$ is bounded from below. The data in Figure~\ref{fig:data} summarizes the behavior of $k d^{-k} \E[A_{d,k}]$ to the maximum $k$ we could easily simulate, $k=18$. For $d=4$, the slow
growth of $A_{d,k}$ makes us suspect that the model is transient.
The plot for $d=2$ confirms \thref{thm:2tree}~(\ref{d2}). Interestingly, $d=3$ appears to be recurrent but very near criticality. The different growth for $d=3$ is grounds for further investigation: it is possible $d=2$ and $d=3$ exhibit different forms of recurrence.

For $d=2$ the simulated values of $k2^{-k} \E[A_{2,k}]$ appear to be growing linearly. This suggests a constant expected number of returns between each successive stunning. As the average number of steps for an individual frog between stunnings is constant, this could indicate that the average time between returns is also bounded away from infinity. If this is the case then the probability that there is a frog at the origin at time $t$ would be bounded away from 0 as $t$ gets large. However, for $d=3$ it appears that $k3^{-k} \E[A_{3,k}]$ is sublinear. This might indicate that the average time between returns is unbounded and the probability of a frog occupying the origin at time $t$ is approaching 0 as $t$ approaches infinity.
This leads us to ask the following:

\begingroup
\def\thethm{\ref{q:strong_weak}}
\begin{question}
Is the frog model strongly recurrent on $\TT_2$, but only weakly recurrent on $\TT_3$?
\end{question}
\addtocounter{thm}{-1}
\endgroup
Such a result would have analogues with other interacting particle systems on trees. For example percolation on 
$\TT_6 \times \Z$ has a three phases: no infinite components, infinitely many infinite components and a unique infinite component \cite{grimmett1990percolation}.
Similarly the contact process on trees can have strongly recurrent, weakly recurrent, and extinction 
phases \cite{pemantle92}.

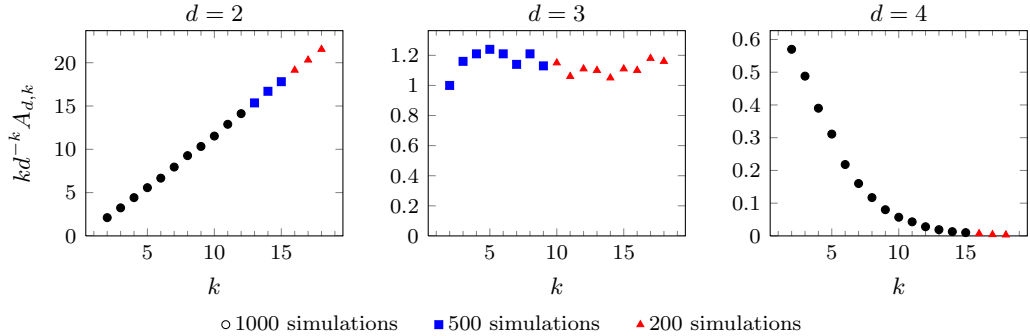
\begin{figure}
  \begin{center}
  \begin{tikzpicture}
    \begin{axis}[
      title={$d = 2$},
      xlabel={$k$},
      ylabel={$kd^{-k}A_{d,k}$},
      ymin=0, footnotesize,xtick={0,5,...,20},
      minor x tick num=4
      ]
      \addplot[only marks, mark size=1.5pt] 
        table[x index=0, y index = 1] {
 2    2.11
 3    3.24
 4    4.42
 5    5.57
 6    6.67
 7    7.95
 8    9.27
 9   10.33
10   11.53
11   12.89
12   14.12
};
      \addplot[only marks, mark size=1.5pt,only marks, blue, mark=square*, mark size=1.5pt] 
        table[x index=0, y index = 1] {
13   15.36
14   16.70
15   17.81
};
      \addplot[only marks, mark size=1.5pt,only marks, red, mark=triangle*, mark size=1.5pt] 
        table[x index=0, y index = 1] {
16   19.14
17   20.31
18   21.55
};
    \end{axis}
    
  \end{tikzpicture}\hspace{0.5cm}%
  \begin{tikzpicture}
    \begin{axis}[
      title={$d = 3$},
      xlabel={$k$},
      ymin=0, footnotesize,xtick={0,5,...,20},
      minor x tick num=4,
      legend style={draw=none},
      legend entries={1000 simulations\hspace*{12pt},500 simulations\hspace*{12pt},%
                      200 simulations},
      legend columns=-1,
      legend to name={leg}
      ]
      \addplot[only marks, mark size=1.5pt] table[x index=0, y index=1] {
0     -1      
      };
      \addplot[only marks, blue, mark=square*, mark size=1.5pt]  table[x index=0, y index = 1] {
 2    1.00
 3    1.16
 4    1.21
 5    1.24
 6    1.21
 7    1.14
 8    1.21
 9    1.13
 };
       \addplot[only marks, red, mark=triangle*, mark size=1.5pt]  table[x index=0, y index = 1] {
10    1.15
11    1.06
12    1.11
13    1.10
14    1.05
15    1.11
16    1.10
17    1.18
18    1.16
};
    \end{axis}    
  \end{tikzpicture}\hspace{0.5cm}%
  \begin{tikzpicture}
    \begin{axis}[
      title={$d = 4$},
      xlabel={$k$},
      ymin=0, footnotesize,xtick={0,5,...,20},
      minor x tick num=4
      ]
      \addplot[only marks, mark size=1.5pt]
         table[x index=0, y index = 1] {
 2   0.570
 3   0.488
 4   0.390
 5   0.311
 6   0.218
 7   0.160
 8   0.117
 9   0.080
10   0.057
11   0.043
12   0.028
13   0.019
14   0.013
15   0.010
};
      \addplot[only marks, mark size=1.5pt, mark=triangle*,red]
         table[x index=0, y index = 1] {
16   0.007      200
17   0.004      200
18   0.003      200
};
    \end{axis}    
  \end{tikzpicture}
  \ref{leg}
  \end{center}
  \caption{Plots of simulated values of $kd^{-k}A_{d,k}$ against $k$ for $d=2,3,4$.
  The number of simulations used in each estimate is shown in the chart.}
  \label{fig:data}
\end{figure}







\subsubsection*{Acknowledgments}

We would like to thank Shirshendu Ganguly for his suggestions throughout the project. Christopher Fowler helped with a calculation, Avi Levy asked a question which led to the inclusion of \thref{lem:preserveincreasing} and James Morrow helped address potential concerns about roundoff error. Thanks to Soumik Pal who pointed out for large $d$ the dynamics should be simpler---this remark sparked our study of transience. We thank Robin Pemantle for directing our attention to \cite{AB}.
We also thank Nina Gantert who pointed out an inaccuracy in a previous version.

The first author was partially supported by NSF grant DMS-1308645 and NSA grant H98230-13-1-0827,
the second author by NSF CAREER award DMS-0847661 and NSF grant DMS-1401479, and the
 third author by NSF RTG grant 0838212.

\bibliographystyle{amsalpha}
\bibliography{frog_paper2}

\end{document}